\documentclass[11pt, letterpaper]{amsart}
\usepackage{latexsym,amssymb,graphicx,multicol,mathrsfs,color,xy,amscd, amsxtra, verbatim, esdiff, enumerate, tikz}

\theoremstyle{plain}

\newtheorem{proposition}{Proposition}[section]
\newtheorem{corollary}[proposition]{Corollary}
\newtheorem*{proposition*}{Proposition}

\newtheorem{lemma}[proposition]{Lemma}

\newtheorem*{theorem*}{Theorem}
\newtheorem{theorem}{Theorem}

\newtheorem*{definition}{Definition}

\newtheorem{innercustomthm}{Theorem}
\newenvironment{customthm}[1]
  {\renewcommand\theinnercustomthm{#1}\innercustomthm}
  {\endinnercustomthm}

\input xy
\xyoption{all}

\def\A{\mathcal{A}}

\def\dim{\text{dim\,}}

\def\C{\mathcal{C}}

\def\dim{\text{\rm dim\,}}

\def\Exc{\text{Exc\,}}

\def\L{\mathscr{L}}
\def\P{\mathbb{L}}
\def\M{\overline{M}}
\def\N{\mathbb{N}}
\def\SM{\overline{\mathcal{M}}}

\def\O{\mathcal{O}}

\def\P{\mathbb{P}}
\def\Q{\mathbb{Q}}
\def\X{\mathcal{X}}

\def\sigman{\{\sigma_{i}\}_{i=1}^{n}}

\def\sigmapn{\{\sigma'_{i}\}_{i=1}^{n}}

\def\S{\mathbb{S}}

\def\Proj{\text{\rm Proj\,}}

\def\Pic{\text{\rm Pic\,}}

\newcommand{\pder}[3][]{\frac{\partial^{#1}#2}{\partial#3^{#1}}}
\newcommand{\dop}[2][]{\frac{\partial^{#1}}{\partial#2^{#1}}}

\begin{document}

\title[Intersection numbers on $\M_{1,n}(m)$]{Intersections of $\psi$-classes on $\M_{1,n}(m)$}
\author{David Ishii Smyth}

\maketitle
\begin{abstract}
We explain how to compute top-dimensional intersections of $\psi$-classes on $\M_{1,n}(m)$, the moduli space of $m$-stable curves. On the spaces $\M_{1,n}$,  these intersection numbers are determined by two recursions, namely the string equation and dilaton equation. We establish, for each fixed $m \geq 1$, an analogous pair of recursions that determine these intersection numbers on the spaces $\M_{1,n}(m)$.
\end{abstract}

\tableofcontents

\section{Introduction}

In \cite{Witten}, Witten made a remarkable conjecture concerning intersections of $\psi$-classes on $\M_{g,n}$. To recall the statement, let $\pi: \C \rightarrow \SM_{g,n}$ denote the universal curve over the moduli stack of stable curves, let $\{\sigma_i\}_{i=1}^{n}$ denote the universal sections of $\pi$, and let $L_i:=\sigma_i^{*}\omega_{\pi}.$ Then $L_i$ descends to a $\Q$-line bundle on the coarse moduli space $\M_{g,n}$, and we define
$$
\psi_i:=c_1(L_i) \in A^{1}_{\Q}(\M_{g,n}).
$$
For any collection of nonnegative integers $d_1, \ldots, d_n$ satisfying $\sum_{i=1}^{n}d_i= \dim \M_{g,n}=3g-3+n$, we let
$$
\langle \psi_1^{d_1}\psi_2^{d_2}\ldots\psi_n^{d_n} \rangle_{g,n} \in \Q
$$
denote the degree of the class $\psi_1^{d_1}\psi_2^{d_2}\ldots\psi_n^{d_n} \in A^{3g-3+n}_{\Q}(\M_{g,n})$. We will call these rational numbers \emph{Witten-Kontsevich numbers}. Witten assembled these numbers into a generating function, conjectured that this function should solve a certain system of partial differential equations (the so-called KdV hierarchy), and showed that the resulting recursions would determine the numbers uniquely. The conjecture was proven by Kontsevich \cite{Kontsevich}, as well as Okounkov-Pandharipande \cite{okounkov-pandharipande} and Mirzakhani \cite{mirzakhani}, and was a major impetus for the development of Gromov-Witten theory.

Evidently, the definition of these intersection numbers depends not only on $M_{g,n},$ but on the specific choice of compactification $\M_{g,n}$. Through recent work on the Hassett-Keel program, we now know that there are many geometrically meaningful compactifications of $M_{g,n}$. Specifically, for any rational number $\alpha \in \Q \cap [0,1]$, we expect the log-canonical model
$$
\M_{g,n}(\alpha):=\Proj R(\SM_{g,n}, K_{\SM_{g,n}}+\alpha\delta+(1-\alpha)\psi)
$$
to have a modular interpretation as a moduli space of Gorenstein pointed curves \cite{afs}, so one should be able to define $\psi_i \in A^{1}_{\Q}(\M_{g,n}(\alpha))$ exactly as above, and define the intersection number $$
\langle \psi_1^{d_1}\psi_2^{d_2}\ldots\psi_n^{d_n} \rangle^{\alpha}_{g,n} \in \Q
$$
as the degree of  $\psi_1^{d_1}\psi_2^{d_2}\ldots\psi_n^{d_n}$. It is natural to ask whether the corresponding generating function is again a solution for the KdV hierarchy, or some other integrable system.

In this paper, we take a step toward answering this question by explaining how to compute these invariants in genus one (for all $\alpha$). Since the Hassett-Keel program (with the scaling defined above) is trivial in genus zero (i.e. $\M_{0,n}(\alpha)=\M_{0,n}$ for all $\alpha$), this is the first case in which genuinely new invariants arise.\footnote{
We should remark that if one adopts the scaling $\M_{g,n}(\alpha):=\Proj R(\SM_{g,n}, K_{\SM_{g,n}}+\alpha\delta)$, then the log canonical models $\M_{0,n}(\alpha)$ are each isomorphic to a moduli space of weighted pointed stable curves $\M_{0,\A}$ \cite{fedorchuk-smyth_lcmodels, alexeev-swinarski}, and top intersections of $\psi$-classes on these spaces are computed in \cite{guy-alexeev}. However, this choice scaling is not compatible with the natural boundary stratification of $\M_{g},$ and therefore not the appropriate generalization of the Hassett-Keel program to $\M_{g,n}$.} To begin, let us recall how the Witten-Kontsevich numbers are computed for $\M_{1,n}$. This does not require the full strength of the Witten-Kontsevich theorem, but only the following proposition, proved in Witten's original paper.

\begin{proposition*}[\cite{Witten}] The Witten-Kontsevich numbers satisfy the following two recursions.
\begin{enumerate}[         (a)]
\item (String Equation)
Suppose $d_1, \ldots, d_n$ satisfy $\sum_{i=1}^{n} d_i=3g-3+n+1.$ Then
\begin{align*}
\langle \prod_{i=1}^{n} \psi_i^{d_i} \cdot \psi_{n+1}^{0} \rangle_{g,n+1}&=\sum_{j=1}^{n}\langle\prod_{i=1}^{n}\psi_i^{d_i-\delta_{ij}} \rangle_{g,n}\\
\end{align*}

\item (Dilaton Equation)
Suppose that $d_1, \ldots, d_n$ satisfy $\sum_{i=1}^{n} d_i=3g-3+n.$ Then
\begin{align*}
\langle \prod_{i=1}^{n} \psi_i^{d_i} \cdot \psi_{n+1} \rangle_{g,n+1}&=(2g-2+n) \langle\prod_{i=1}^{n}\psi_i^{d_i} \rangle_{g,n}
\end{align*}
\end{enumerate}
\end{proposition*}
These recursions allow one to compute a Witten-Kontsevich number on $\M_{g,n+1}$ as a sum of Witten-Kontsevich numbers on $\M_{g,n}$, provided that at least one $\psi_i$ appears with multiplicity zero or one. Since any top-dimensional intersection product $\psi_1^{d_1}\ldots\psi_n^{d_n}$ on $\M_{1,n}$ satisfies $\sum_{i=1}^{n} d_i=\dim \M_{1,n}=n,$ we necessarily have $d_i=0$ or $1$ for at least one $i$. Thus, all genus one invariants can be computed inductively starting from the single, well-known initial condition $\langle \psi_1 \rangle_{1,1}=1/24$. 

Now we explain how this picture generalizes to the log-canonical models $\M_{1,n}(\alpha).$ In \cite{smyth_elliptic1}, we constructed for each pair of integers $n > m \geq 1$, a Deligne-Mumford stack $\SM_{1,n}(m)$ parametrizing $n$-pointed elliptic curves with nodes and elliptic $k$-fold points $(k=1,2,\ldots, m)$ as allowable singularities. In \cite{smyth_elliptic2}, we showed that these stacks have projective coarse moduli spaces, and that each log canonical model $\M_{1,n}(\alpha)$ is isomorphic to one of the coarse moduli spaces $\M_{1,n}(m)$ (we adopt the convention that $\M_{1,n}(0):=\M_{1,n}$).  Thus, describing all genus one invariants $\langle \psi_1^{d_1}\psi_2^{d_2}\ldots\psi_n^{d_n} \rangle_{1,n}^{\alpha}$  for a specified $\alpha$ is equivalent to computing $\langle \psi_1^{d_1}\psi_2^{d_2}\ldots\psi_n^{d_n} \rangle^{m}$ for a specified $m$, where 
$$
\langle \psi_1^{d_1}\psi_2^{d_2}\ldots\psi_n^{d_n} \rangle^{m} \in \Q
$$
denotes the degree of  $\psi_1^{d_1}\psi_2^{d_2}\ldots\psi_n^{d_n}$ in $A_{\Q}^{*}(\M_{1,n}(m))$. We call these intersection numbers \emph{$m$-stable Witten-Kontsevich numbers}. The results of this paper determine, for each $m,$ a pair of recursions and an initial condition which determine all $m$-stable Witten-Kontsevich numbers. These new recursions differ from the original string/dilaton equations by a sum of ``error" terms, which are naturally indexed by $m$-partitions of $[n]$, i.e. partitions of $[n]:=\{1, \ldots, n\}$ into $m$ disjoint, nonempty subsets.  In order to give a precise statement, we introduce some additional notation. 

We let ${n \brack m}$ denote the set of all $m$-partitions of $[n]$. If $S=\{S_1, \ldots, S_m\}$ is an $m$-partition of $[n]$, we define $k(S)$ to be the number of $S_i$ such that $|S_i| \geq 2$, and we alway assume that the $S_i$ are labelled so that $S_1, \ldots, S_{k(S)}$ satisfy $|S_i| \geq 2$, and $S_{k(S)+1}, \ldots, S_{m}$ are singletons. We call $k(S)$ the \emph{index of $S$}. Finally, if $n, a_1, \ldots, a_l$ are nonnegative integers, we set
\begin{align*}
{n \choose a_1 \ldots a_l}:=
\frac{n!}{a_1!a_2!\ldots a_l! (n-\sum_{i=1}^{l} a_i)!}
\end{align*}
provided $\sum_{i=1}^{l} a_i \leq n,$ and zero otherwise. We adopt the usual convention that $0!=1$. We can now state our main results.

\begin{theorem} Fix nonnegative integers $n, m$ satisfying $n>m$. The $m$-stable Witten-Kontsevich numbers satisfy the following two recursions.
\begin{enumerate}[         (a)]
\item ($m$-stable String Equation) Suppose $d_1, \ldots, d_n$ satisfy $\sum_{i=1}^{n} d_i=n+1.$ Then
\begin{align*}
\langle \prod_{i=1}^{n} \psi_i^{d_i} \cdot \psi_{n+1}^{0} \rangle^m&=\sum_{j=1}^{n}\langle \prod_{i=1}^{n}\psi_i^{d_i-\delta_{ij}} \rangle^m + \frac{m! }{24} \sum_{S \in {n \brack m}}(-1)^{\star(S)}\prod_{j=1}^{k(S)}{|S_j|-1 \choose \{d_i\}_{i \in S_j}}
\end{align*}
\item ($m$-stable Dilaton Equation)
Suppose $d_1, \ldots, d_n$ satisfy $\sum_{i=1}^{n} d_i=n.$ Then
\begin{align*}
\langle \prod_{i=1}^{n} \psi_i^{d_i} \cdot \psi_{n+1} \rangle^m &=n \langle\prod_{i=1}^{n}\psi_i^{d_i} \rangle^m +\frac{m! }{24} \sum_{S \in {n \brack m}}(-1)^{\star(S)} \prod_{j=1}^{k(S)}{|S_j|-1 \choose \{d_i\}_{i \in S_j}}
\end{align*}
where $\star(S):={n-m-k(S)-\sum_{j \in S_1 \cup \ldots \cup S_{k(S)}} d_j}-1.$
\end{enumerate}
\end{theorem}

The original string/dilaton recursions are proved by analyzing the behavior of $\psi$-classes under the forgetful morphism $\SM_{g,n+1} \rightarrow \SM_{g,n}$. Our modified recursions are proved similarly by analyzing the behavior of $\psi$-classes under the \emph{rational} forgetful map $\SM_{1,n+1}(m) \dashrightarrow \SM_{1,n}(m)$. The error terms correspond to certain intersection numbers supported on the exceptional divisors of a resolution of this rational map.

By analyzing the rational reduction map $\SM_{1,n}(m) \dashrightarrow \SM_{1,n}(m+1),$ one can also prove a recursion relating $(m+1)$-stable and $m$-stable Witten-Kontsevich numbers. (There is no analogue of this recursion in Witten's original paper.)

\begin{theorem}[Reduction Recursion] Fix nonnegative integers $n, m$ satisfying $n>m+1$. Suppose $d_1, \ldots, d_n$ satisfy $\sum d_i=n.$ Then
\begin{align*}
\langle \prod_{i=1}^{n} \psi_i^{d_i} \rangle^{m+1} &=\langle\prod_{i=1}^{n}\psi_i^{d_i} \rangle^{m} +\frac{m! }{24} \sum_{S \in {n \brack m+1}}(-1)^{\star(S)} \prod_{j=1}^{k(S)}{|S_j|-1 \choose \{d_i\}_{i \in S_j}}
\end{align*}
where $\star(S):={n-m-k(S)-\sum_{j \in S_1 \cup \ldots \cup S_{k(S)}} d_j}-1.$
\end{theorem}

Theorem 1 allows one to compute a Witten-Kontsevich number on $\M_{1,n+1}(m)$ in terms of Witten-Kontsevich numbers on $\M_{1,n}(m)$ provided $n>m$. Thus, one can recursively compute all $m$-stable Witten-Kontsevich numbers in terms of Witten-Kontsevich numbers on $\M_{1,m+1}(m)$. Using Theorems 1 and 2 together, we can inductively compute the latter, thus establishing the necessary initial condition for the hierarchy of $m$-stable Witten-Kontsevich numbers.

\begin{theorem}[$m$-stable Initial Condition] Fix a nonnegative integer $m$.
Suppose $d_1, \ldots, d_{m+1}$ satisfy $\sum_{i=1}^{m+1}d_i=m+1$. Then we have
$$
\langle \psi_1^{d_1}\psi_2^{d_2}\ldots\psi_{m+1}^{d_{m+1}} \rangle^{m}=\frac{m !}{24}.
$$
\end{theorem}

Theorem 3 should be viewed as the $m$-stable analogue of the initial condition $\langle \psi_1 \rangle_{1,1}=1/24$ on $\M_{1,1}$. Evidently, Theorems 1 and 3 taken together completely determine all $m$-stable Witten-Kontsevich numbers. At first glance, it may appear strange that the result of Theorem 3 does not depend on $d_1, \ldots, d_{m+1}$. The reason for this is that the $\Q$-Picard number of $\M_{1,m+1}(m)$ is one, so that $\psi_1=\psi_2=\ldots=\psi_{m+1} \in A^1_{\Q}(\M_{1,m+1}(m))$. 

We should make a remark concerning the content of Theorems 1 and 2 when $m=0, 1$. Theorem 1 is valid when $m=0$ since ${n \brack 0} = \emptyset$, so there are no error terms. Theorem 1 is also valid when $m=1,$ but potentially misleading. In this case, there is exactly one error term (since there is unique partition in $n \brack 1$, namely $[n]$ itself), but this error term is always zero. Indeed, we have $${|S_1|-1 \choose \{d_i\}_{i \in S_1}}=0,$$ since $|S_1|-1=n-1$, but $\sum_{i \in S_1} d_i =\sum_{i=1}^{n}d_i=n+1$ (resp. $n$) in case $(a)$ (resp. (b)). The same reasoning shows that the error term in Theorem 2 also vanishes when $m=1$, i.e. $\langle \prod_{i=1}^{n} \psi_i^{d_i} \rangle^{1} =\langle\prod_{i=1}^{n}\psi_i^{d_i} \rangle^{0}$. In other words, the 1-stable Witten-Kontsevich numbers are identical to the ordinary Witten-Kontsevich numbers and satisfy the same recursions. This can be understood as a consequence of the fact that the natural birational map $\phi: \M_{1,n} \rightarrow \M_{1,n}(1)$ is regular and satisfies $\phi^*\psi_i=\psi_i$. The analogous statement is not true for $m > 1$, and one sees genuinely new invariants for all $m \geq 2$. 

The rest of this paper is organized as follows. In Section 2, we describe the indeterminacy loci of the forgetful and reduction maps. We show that these maps can be resolved by a simple blow-up, and compare pull-backs of $\psi$-classes as a sum of exceptional divisors on this resolution. In Section 3, we prove our main results. The general shape of Theorems 1 and 2 follows easily from the comparison formulae of Section 2 and the push-pull formula, but computing the error terms explicitly requires an elaborate calculation in the Chow ring of the resolution (Section \ref{S:intersection}). In Section 4, we show how to use Theorems 1 and 2 in several sample calculations, and provide a reference table of all $m$-stable Witten-Kontsevich numbers with $m<n \leq 6$ (Table 1).

\emph{Acknowledgements.} The author thanks Yaim Cooper, Norman Do, Maksym Fedorchuk, Paul Norbury, and Aaron Pixton for conversations related to this work.

\section{Resolution of Forgetful and Reduction Maps}

In this section, we construct resolutions of the natural forgetful and reduction maps between the stacks $\SM_{1,n}(m)$. We should remark that Ranganathan, Santos-Parker, and Weiss recently constructed a resolution of the rational map $\SM_{1,n} \dashrightarrow \SM_{1,n}(m)$ using ideas from tropical geometry \cite{rspw1, rspw2}. For our purpose, it is more convenient to go step-by-step, i.e. to consider $\SM_{1,n}(m) \dashrightarrow \SM_{1,n}(m+1)$, so we give an independent argument.

\subsection{Forgetful Map}\label{S:forgetful}
Fix positive integers $m <n$, and consider the rational map
$$F: \SM_{1,n+1}(m) \dashrightarrow \SM_{1,n}(m)$$
obtained by forgetting the $(n+1)^{st}$ marked point. It is immediate from the definition of $m$-stability that if $(C, \{p_i\}_{i=1}^{n+1})$ is $m$-stable, then $(C, \{p_i\}_{i=1}^{n})$ remains $m$-stable unless one of the following conditions holds.
\begin{enumerate}
\item $C$ contains a smooth rational component with three distinguished points, one of which is $p_{n+1}$.
\item $C$ contains an elliptic spine with $m+1$ distinguished points, one of which is $p_{n+1}$. \footnote{Here, an elliptic spine is simply an arithmetic genus one subcurve with no disconnecting nodes, and a distinguished point is simply a marked point or a disconnecting node  (see \cite[Definition 2.9]{smyth_elliptic1}) .}
\end{enumerate}
It follows that $F$ is regular away from the locus of curves satisfying (1) or (2). Of course, it is well-understood how to ``stabilize" a curve $(C, \{p_i\}_{i=1}^{n})$ in case (1); one simply contracts a destabilizing rational tail/bridge to a smooth/nodal point. Furthermore, this stabilization can be carried out simultaneously on the fibers of the universal family $\pi: \C \rightarrow \SM_{1,n+1}(m)$ by taking the map associated to a high power of $\omega_{\pi}(\Sigma_{i=1}^{n}\sigma_i)$.  It follows that $F$ is in fact regular away from the locus of curves satisfying (2).

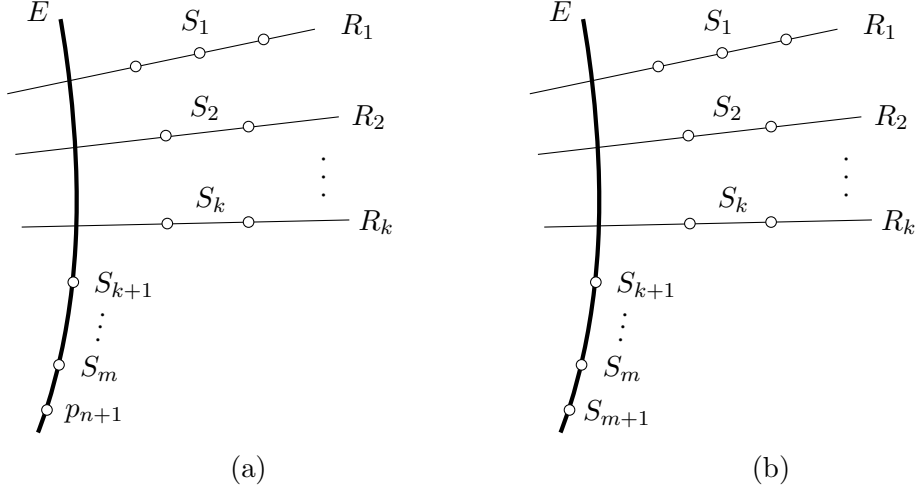
\begin{figure}

\centering
\begin{tikzpicture}[scale=1]
\node at (1.5,4.1) {$E$};
\draw[ultra thick] plot [smooth, tension=1] coordinates {(1.8,4)  (2,1) (1.5,-1.5)} ;
\draw[rotate=5] plot [smooth, tension=1] coordinates {(1.35,2.9)  (5.5,3.4)};
\node at (3.6,3.96) {$S_{1}$};
\node at (5.75,3.9) {$R_1$};
\draw[fill=white] (4.5,3.73) circle(2pt);
\draw[fill=white] (3.65,3.55) circle(2pt);
\draw[fill=white] (2.8,3.37) circle(2pt);
\draw plot [smooth, tension=1] coordinates {(1.2,2.2)  (5.5,2.7)};
\node at (3.72,2.82) {$S_{2}$};
\node at (5.9,2.7) {$R_2$};
\draw[fill=white] (4.3,2.57) circle(2pt);
\draw[fill=white] (3.2,2.45) circle(2pt);
\node[rotate=90] at (5.3,1.9) {$.$ $.$ $.$};
\draw[rotate=-8] plot [smooth, tension=1] coordinates {(1.1,1.4)  (5.4,2.1)};
\node at (3.8,1.6) {$S_{k}$};
\node at (6,1.3) {$R_k$};
\draw[fill=white] (4.3,1.3) circle(2pt);
\draw[fill=white] (3.22,1.28) circle(2pt);

\draw[fill=white] (1.97,.5) circle(2pt);
\node at (2.64,.45) {$S_{k+1}$};
\node[rotate=80] at (2.33,-.097) {$.\!$ $.\!$ $.$};
\draw[fill=white] (1.78,-.6) circle(2pt);
\node at (2.32,-.65) {$S_{m}$};
\draw[fill=white] (1.62,-1.2) circle(2pt);
\node at (2.25,-1.25) {$p_{n+1}$};
\node at (4.3,-2){(a)};
\end{tikzpicture}
\qquad\qquad
\begin{tikzpicture}[scale=1]
\node at (1.5,4.1) {$E$};
\draw[ultra thick] plot [smooth, tension=1] coordinates {(1.8,4)  (2,1) (1.5,-1.5)} ;
\draw[rotate=5] plot [smooth, tension=1] coordinates {(1.35,2.9)  (5.5,3.4)};
\node at (3.6,3.96) {$S_{1}$};
\node at (5.75,3.9) {$R_1$};
\draw[fill=white] (4.5,3.73) circle(2pt);
\draw[fill=white] (3.65,3.55) circle(2pt);
\draw[fill=white] (2.8,3.37) circle(2pt);
\draw plot [smooth, tension=1] coordinates {(1.2,2.2)  (5.5,2.7)};
\node at (3.72,2.82) {$S_{2}$};
\node at (5.9,2.7) {$R_2$};
\draw[fill=white] (4.3,2.57) circle(2pt);
\draw[fill=white] (3.2,2.45) circle(2pt);
\node[rotate=90] at (5.3,1.9) {$.$ $.$ $.$};
\draw[rotate=-8] plot [smooth, tension=1] coordinates {(1.1,1.4)  (5.4,2.1)};
\node at (3.8,1.6) {$S_{k}$};
\node at (6,1.3) {$R_k$};
\draw[fill=white] (4.3,1.3) circle(2pt);
\draw[fill=white] (3.22,1.28) circle(2pt);

\draw[fill=white] (1.97,.5) circle(2pt);
\node at (2.64,.45) {$S_{k+1}$};
\node[rotate=80] at (2.33,-.097) {$.\!$ $.\!$ $.$};
\draw[fill=white] (1.78,-.6) circle(2pt);
\node at (2.32,-.65) {$S_{m}$};
\draw[fill=white] (1.62,-1.2) circle(2pt);
\node at (2.25,-1.25) {$S_{m+1}$};
\node at (4.3,-2){(b)};
\end{tikzpicture}
\caption{Topological type of a generic curve in (a) $\Delta_S \subset \SM_{1,n+1}(m)$ (see Section \ref{S:forgetful}), and (b)  $\Delta_S \subset \SM_{1,n}(m)$ (see Section \ref{S:reduction}) } \label{F1}
\end{figure}
Now we consider the problem of stabilizing curves in case (2). We can describe this locus of curves as follows. For each $S \in {n \brack m},$ define $\Delta_{S} \subset \SM_{1,n+1}(m)$ to be the closed substack
$$
\Delta_{S}:=\cap_{i=1}^{k}\Delta_{0,S_i},
$$
where $k:=k(S)$ is the index of $S$. Equivalently,
 $$\Delta _S \simeq \SM_{1,m+1}(m) \times \SM_{0,|S_1|+1} \times \cdots \times  \SM_{0,|S_k|+1}$$
 is the boundary stratum parametrizing $m$-stable curves with $k$ rational tails, marked by $S_1, \ldots, S_k$, and an elliptic spine, marked by $S_{k+1} \cup \ldots \cup S_m \cup \{p_{n+1}\}$ (See Figure 1(a)). Note that these boundary strata are pairwise disjoint in $\SM_{1,n+1}(m)$ (this is an easy consequence of the fundamental decomposition of an $m$-stable curve \cite[Lemma 3.1]{smyth_elliptic1}), and the locus of curves satisfying (2) is precisely  $\cup_{S \in  {n \brack m}} \Delta_{S}.$

Now it is clear that if $(C, \{p_i\}_{i=1}^{n+1}) \in \cup_{S \in  {n \brack m}} \Delta_{S}$, then stabilizing $(C, \{p_i\}_{i=1}^{n})$ should entail contracting the elliptic $m$-spine (created by forgetting $p_{n+1}$) to an elliptic $m$-fold point. However, this requires a choice of moduli of attaching data for the elliptic $m$-fold point (see \cite[Section 2.2]{smyth_elliptic2}), which means we cannot stabilize the universal family $\C \rightarrow \SM_{1,n+1}(m)$ without a birational modification of the base. Happily, the required modification is as simple as it could be. Indeed, let $p: \X \rightarrow \SM_{1,n+1}(m)$ be the blow-up of $\SM_{1,n+1}(m)$ along $\cup_{S \in {n \brack m}}\Delta_{S},$ and let $\C \rightarrow \X$ be the pull-back of the universal family. We will show that it is possible to stabilize $\C \rightarrow \X$, so that there is an induced regular map $\X \rightarrow \SM_{1,n}(m)$. 

We will need the following bit of notation: For each $S \in {n \brack m},$ let $E_S \subset \X$ denote the exceptional divisor lying over $\Delta_{S}.$ Also, if $I \subset [n]$ is any index set, we let $\Delta_{I} \subset \X$ denote the strict transform of $\Delta_{0,I} \subset \SM_{1,n+1}(m)$. Finally, if $\C \rightarrow \X$ is the pullback of the universal curve over $\SM_{1,n+1}(m)$, note that there is a unique irreducible Cartier divisor in $\C$ which comprises the elliptic spines of the fibers of $\C|_{E_S} \rightarrow E_S$; we will call this Cartier divisor $E_{S}^{1}$.

\begin{proposition}[Resolution of the forgetful map]\label{P:forgetful}
With notation as above, consider the commutative diagram
\[
\xymatrix{ 
&\tilde{\C} \ar[rd]^{c} \ar[ld]_{b} \ar[dd]^{\tilde{\pi}}&\\
\C  \ar[rd]^{\pi}&& \C' \ar[ld]_{\pi'}\\
&\X  \ar[ld]_{p} \ar[rd]^{q} \ar@/^1pc/[lu]^{\sigman} \ar@/_1pc/[ru]_{\sigmapn}  &\\
\SM_{1,n+1}(m) \ar@{-->}[rr]&&\SM_{1,n}(m)\\
}
\]
where
\begin{itemize}
\item[(1)] $\left( \pi, \{\sigma_i\}_{i=1}^{n} \right)$ is the pull-back of the universal family from $\SM_{1,n+1}(m)$ to $\X$.
\item[(2)] $b$ is the blow-up of $\C$ along the smooth codimension-two locus $\cup_{S \in {n \brack m}} \cup_{i=k(S)+1}^{m}(\sigma_i \cap  E_S^1),$ and $\tilde{\sigma}_i, \tilde{E}_{S}, \tilde{E}^1_S$ are the strict transforms of $\sigma_i, E_S, E^1_S$.
\item[(3)]$c$ is the birational contraction associated to a high power of 
$$
\L=\omega_{\tilde{\pi}}\left(\Sigma_{i=1}^{n}\tilde{\sigma_i}+\Sigma_{S \in {n \brack m}}\tilde{E}^1_S\right)
$$ \
and $\sigma_i':=c \circ \sigma_i$ for $i=1, \ldots, n$.
\end{itemize}
Then $(\C' \rightarrow \X, \{\sigma_i'\}_{i=1}^{n})$ is a flat family of $m$-stable, $n$-pointed curves. In particular, there is an induced regular map $q: \X \rightarrow \SM_{1,n}(m)$.
\end{proposition}

\begin{proof}
As in the statement, let $\C \rightarrow \X$ be the pullback of the universal family over $\SM_{1,n+1}(m)$, let $b: \tilde{\C} \rightarrow \C$ be the blowup of $\C$, and consider the line-bundle
$$
\L:=\omega_{\tilde{\pi}}\left(\sum_{i=1}^{n}\tilde{\sigma_i}+\sum_{S \in {n \brack m}}\tilde{E}^1_S\right)
$$ 
We claim that $\L^m$ is $\tilde{\pi}$-semiample for $m>>0$, so that we have a morphism
\[
\xymatrix{\tilde{\C} \ar[rr]^{c} \ar[rd]_{\tilde{\pi}} &&\C' \ar[ld]^{\pi'} &\!\!\!\!\!\!\!\!\!\!\!\!\!\!\!\!:=\Proj \left( \oplus_{m \geq 0} \tilde{\pi}_*\L^m \right)\\
&\X&&
}
\]
This follows from the proof of Lemma 2.12 in \cite{smyth_elliptic1}. Indeed, the argument given there shows that $H^1(\tilde{C}_x, L_x)=0$ for each geometric point $x \in \X$, and the statement follows easily. It only remains to prove that $\C'/\X$ is a flat family of $n$-pointed $m$-stable curves. 

To see this, first note that away from $\cup_{S \in {n \brack m}}\tilde{E}_{S}$, $\L$ is just the standard twisted dualizing sheaf, and thus has the effect of contracting semistable rational tails. Thus, we only need to check that the fibers of $\C'$ are $m$-stable over $\cup_{S \in {n \brack m}}\tilde{E}_{S}$. To see this, consider any map $\Delta \rightarrow \X$ (where $\Delta$ is the spectrum of a DVR) sending the generic point into $p^{-1}(\mathcal{M}_{1,n+1})$, and the closed point into one of the divisors $\tilde{E}_{S}$. Since $H^1(C_x,\L_x)=0$ for all geometric points $x \in \X$, push-forward commutes with base change, and we have that $\C'|_{\Delta}$ is the image of the map associated to a large power of $\L|_{(\tilde{\C}|_{\Delta})}.$ Now Lemma 2.12 in \cite{smyth_elliptic1} implies that $\C'|_{\Delta}$ is a flat family of Gorenstein curves over $\Delta$, in which the elliptic spine has been replaced by an elliptic $m$-fold point. It follows that the fibers of $\C' \rightarrow \X$ are all $n$-pointed, $m$-stable curves. Furthermore, since $\X$ is reduced, the valuative criterion for flatness implies that $\C' \rightarrow \X$ is flat.
\end{proof}

The preceding proposition gives the following comparison formula for pull-backs of $\psi_i$ classes from $\SM_{1,n+1}(m)$ and $\SM_{1,n}(m)$.

\begin{corollary}\label{C:psi-forgetful} For $i \in \{1, \ldots, n\}$, we have
\begin{align*}
q^*\psi_i+\Delta_{\{i,n+1\}}&=p^*\psi_i+\sum_{ \substack{S \in {n \brack m}\\ \{i\} \in S}}E_{S}
\end{align*}
in $A_{\Q}^1(\X)$.
\end{corollary}

\begin{proof}
We have
\begin{align*}
p^*\psi_i&=c_1(\sigma_i^*\omega_{\pi})\\
q^*\psi_i&=c_1(\sigma_i'^*\omega_{\pi'})\\
\end{align*}
We compare these two classes as follows. Since $b$ is just a standard blow-up, we have
$$
b^*\omega_{\pi}=\omega_{\tilde{\pi}}(-\Sigma_{S \in {n \brack m}} \Sigma_{i=k(S)+1}^{m} Z_{\{S,i\}}),
$$
where $Z_{\{S,i\}}$ is the exceptional divisor over $\sigma_i \cap E_S \subset \C$. Restricting this equation to $\tilde{\sigma}_i$ gives
\[p^*\psi_i=\tilde{\sigma}_i^*\omega_{\tilde{\pi}}-\sum_{\substack{S \in {n \brack m}\\ \{i\} \in S }}E_{S}.
\eqno{(\dagger)}
\]
On the other hand, because $c$ is the contraction associated to $\L$, we have
$$
c^*\omega_{\pi'}=\omega_{\tilde{\pi}}(D),
$$
where $D$ is a linear combination of $c$-exceptional divisors. The $c$-exceptional are precisely $\tilde{E}_{S}^1$ (for $S \in {n \brack m}$) and $\tilde{R}^1_{\{i,n+1\}}$ (for $i \in [n]$). (Here, $\tilde{R}^1_{\{i,n+1\}}$ denotes the Cartier divisor of distinguished rational tails in the fibers of $\tilde{\C}|_{\Delta_{\{i,n+1\}}} \rightarrow \Delta_{\{i,n+1\}}$). The coefficients of these divisors in $D$ are easily determined by the requirement that $c^*\omega_{\pi'}$ have degree zero on contracted curves, and we obtain
$$
c^*\omega_{\pi'}=\omega_{\tilde{\pi}}(\Sigma_{S \in {n \brack m}}\tilde{E}_{S}^1-\Sigma_{i=1}^{n}\tilde{R}^1_{\{i,n+1\}})
$$
Since $\tilde{\sigma}_i$ never intersects $\tilde{E}_S^1$, restricting this equation to $\tilde{\sigma}_i$ gives
\[
q^*\psi_i=\sigma_i^*\omega_{\tilde{\pi}}-\Delta_{\{i,n+1\}}
\eqno{(\dagger\dagger)}
\]
Combining $(\dagger)$ and $(\dagger \dagger)$ gives the desired result.
\end{proof}

\subsection{Reduction Map}\label{S:reduction}
Fix positive integers $n>m+1$, and consider the natural birational map
$$
R: \SM_{1,n}(m) \dashrightarrow \SM_{1,n}(m+1),
$$
which is well-defined away from the locus of curves containing an elliptic $(m+1)$-spine. This locus can be described as follows:  for each $S \in {n \brack m+1}$, define $\Delta_{S} \subset \SM_{1,n}(m)$ be the locally closed substack
$$
\Delta_{S}:=\cap_{i=1}^{k}\Delta_{S_i}
$$
where $S=\{S_1, \ldots, S_{m+1}\}$ is a partition of index $k$. Equivalently,
 $$\Delta _S \simeq \SM_{1,m+1}(m) \times \SM_{0,|S_1|+1} \times \cdots \times  \SM_{0,|S_k|+1}$$
 is the boundary stratum parametrizing $m$-stable curves with $k$ rational tails, marked by $S_1, \ldots, S_k$, and an elliptic spine, marked by $S_{k+1} \cup \ldots \cup S_{m+1}.$ These substacks are pairwise disjoint in $\SM_{1,n}(m)$, and the locus of curves containing an elliptic $(m+1)$-spine is precisely  $\cup_{S \in  {n \brack m+1}} \Delta_{S}$.

As in the preceding section, the indeterminacy of this rational map is resolved by a simple blow-up of the base. Indeed, let $p: \X \rightarrow \SM_{1,n}(m)$ be the blow-up of $\SM_{1,n}(m)$ along $\cup_{S \in {n \brack m+1}}\Delta_{S},$ let $E_{S}$ denote the exceptional divisor lying over $\Delta_{S}$, and let $E_{S}^1 \subset \C$ denote the Cartier divisor comprising the elliptic $(m+1)$-spines of the fibers of $\C|_{E_S} \rightarrow E_S$. Then we have

\begin{proposition}[Resolution of the reduction map]\label{P:reduction}
With notation as above, consider the commutative diagram
\[
\xymatrix{ 
&\tilde{\C} \ar[rd]^{c} \ar[ld]_{b} \ar[dd]^{\tilde{\pi}}&\\
\C  \ar[rd]^{\pi}&& \C' \ar[ld]_{\pi'}\\
&\X  \ar[ld]_{p} \ar[rd]^{q} \ar@/^1pc/[lu]^{\sigman} \ar@/_1pc/[ru]_{\sigmapn}  &\\
\SM_{1,n}(m) \ar@{-->}[rr]&&\SM_{1,n}(m+1)\\
}
\]
where
\begin{itemize}
\item[(1)] $\left( \pi, \{\sigma_i\}_{i=1}^{n} \right)$ is the pull-back of the universal family from $\SM_{1,n}(m)$ to $\X$.
\item[(2)] $b$ is the blow-up of $\C$ along the smooth codimension-two locus $\cup_{S \in {n \brack m+1}} \cup_{i=k(S)+1}^{m+1}(\sigma_i \cap  E_S),$ and $\tilde{\sigma}_i, \tilde{E}_{S}, \tilde{E}^1_S$ are the strict transforms of $\sigma_i, E_S, E^1_S$.
\item[(3)]$c$ is the birational contraction associated to a high power of 
$$
\L=\omega_{\tilde{\pi}}\left(\Sigma_{i=1}^{n}\tilde{\sigma_i}+\Sigma_{S \in {n \brack m+1}}\tilde{E}^1_S\right)
$$ \
and $\sigma_i':=c \circ \sigma_i$ for $i=1, \ldots, n$.
\end{itemize}
Then $(\C' \rightarrow \X, \{\sigma_i'\}_{i=1}^{n})$ is a flat family of $n$-pointed, $(m+1)$-stable curves. In particular, there is an associated regular map $q: \X \rightarrow \SM_{1,n}(m+1)$.
\end{proposition}
\begin{proof}
Arguing precisely as in the proof of Proposition \ref{P:forgetful}, we see that the map associated to $\L$ contracts elliptic $(m+1)$-spines in the fibers of $\tilde{\C} \rightarrow \X$ by elliptic $(m+1)$-fold points, so that
$$\C':=\Proj \left( \oplus_{m \geq 0} \tilde{\pi}_*\L^m \right)$$
is a flat family of $(m+1)$-stable of curves.
\end{proof}

Arguing exactly as in the proof of Corollary \ref{C:psi-forgetful}, we obtain
\begin{corollary}\label{C:psi-reduction}For $i \in \{1, \ldots, n\}$, we have
\begin{align*}
q^*\psi_i&=p^*\psi_i+\sum_{ \substack{S \in {n \brack m+1}\\ \{i\} \in S}}E_{S}
\end{align*}
in $A^1_{\Q}(\X)$.
\end{corollary}

\section{Proof of Main Results}\label{S:Exceptional}
In this section, we prove Theorems 1, 2, and 3. Unfortunately, the statements of these theorems given in the introduction are not well-suited to the logical structure of our planned proof. (One needs Theorems 1 and 2 to prove Theorem 3, but one also needs Theorem 3 to prove Theorems 1 and 2.) To avoid a circular argument, we must introduce the following tweaked versions of Theorem 1 and 2, in which the constant $m!/24$ has been replaced by the as-yet-undetermined intersection number $\langle \psi_1^{m+1} \rangle^m$.

\begin{customthm}{$1^*$} Fix nonnegative integers $n, m$ satisfying $n>m$. The $m$-stable Witten-Kontsevich numbers satisfy the following two recursions.
\begin{enumerate}[         (a)]
\item ($m$-stable String Equation) Suppose $d_1, \ldots, d_n$ satisfy $\sum_{i=1}^{n} d_i=n+1.$ Then
\begin{align*}
\langle \prod_{i=1}^{n} \psi_i^{d_i} \cdot \psi_{n+1}^{0} \rangle^m&=\sum_{j=1}^{n}\langle \prod_{i=1}^{n}\psi_i^{d_i-\delta_{ij}} \rangle^m + \langle \psi_1^{m+1} \rangle^m \sum_{S \in {n \brack m}}(-1)^{\star(S)}\prod_{j=1}^{k(S)}{|S_j|-1 \choose \{d_i\}_{i \in S_j}}
\end{align*}
\item ($m$-stable Dilaton Equation)
Suppose $d_1, \ldots, d_n$ satisfy $\sum_{i=1}^{n} d_i=n.$ Then
\begin{align*}
\langle \prod_{i=1}^{n} \psi_i^{d_i} \cdot \psi_{n+1} \rangle^m &=n \langle\prod_{i=1}^{n}\psi_i^{d_i} \rangle^m +\langle \psi_1^{m+1} \rangle^m\sum_{S \in {n \brack m}}(-1)^{\star(S)} \prod_{j=1}^{k(S)}{|S_j|-1 \choose \{d_i\}_{i \in S_j}}
\end{align*}
where $\star(S):={n-m-k(S)-\sum_{j \in S_1 \cup \ldots \cup S_{k(S)}} d_j}-1.$
\end{enumerate}
\end{customthm}

\begin{customthm}{$2^*$} Fix nonnegative integers $n, m$ satisfying $n>m+1$. Suppose $d_1, \ldots, d_n$ satisfy $\sum d_i=n.$ Then
\begin{align*}
\langle \prod_{i=1}^{n} \psi_i^{d_i} \rangle^{m+1} &=\langle\prod_{i=1}^{n}\psi_i^{d_i} \rangle^{m} +\langle \psi_1^{m+1} \rangle^m\sum_{S \in {n \brack m+1}}(-1)^{\star(S)} \prod_{j=1}^{k(S)}{|S_j|-1 \choose \{d_i\}_{i \in S_j}}.
\end{align*}
where $\star(S):={n-m-k(S)-\sum_{j \in S_1 \cup \ldots \cup S_{k(S)}} d_j}-1.$
\end{customthm}

Now the logical structure of our argument is as follows. In Section \ref{S:mainproof}, we prove Theorems $1^*$ and $2^*$, making use of a key intersection theory computation in Section 3.2.  In Section 3.3, we use Theorems $1^*$ and $2^*$ to prove Theorem 3. Since Theorem 3 states that $\langle \psi_1^{m+1} \rangle^m=m !/24,$ Theorems 1 and 2 follow immediately.

\subsection{Proof of String/Dilaton and Reduction Recursions}\label{S:mainproof}

The strategy for proving Theorems $1^*$ and $2^*$ is straightforward: we use Corollary \ref{C:psi-forgetful} (resp. Corollary \ref{C:psi-reduction}) to compare products of $\psi$-classes on a resolution of the forgetful (resp. reduction) map. The hard part is obtaining an explicit formula for the contributions arising from the exceptional divisors, and this calculation is carried out in Section \ref{S:intersection}.

\subsubsection{$m$-stable String equation}\label{S:string}

To prove Theorem $1^*\,(a)$, we consider the resolution of the forgetful map from Section \ref{S:forgetful}:

\[
\xymatrix{ 
&\X  \ar[ld]_{p} \ar[rd]^{q} &\\
\SM_{1,n+1}(m)  \ar@{-->}[rr]&&\SM_{1,n}(m)\\
}
\]

Given nonnegative integers $d_1, \ldots, d_n$ satisfying $\sum_{i=1}^{n} d_i=n+1$, Corollary \ref{C:psi-reduction} gives the following equation in $A_{\Q}^{n+1}(\X)$:
\begin{align*}
(q^*\psi_1+\Delta_{\{1,n+1\}})^{d_1}\cdots (q^*\psi_n+\Delta_{\{n,n+1\}})^{d_n}=(p^*\psi_1+\sum_{\substack{S \in {n \brack m}\\ \{1\} \in S }}E_{S})^{d_1}\cdots (p^*\psi_n+\sum_{\substack{S \in {n \brack m}\\ \{n\} \in S }}E_{S})^{d_n}.\\
\end{align*}

First, we show that the degree of the lefthand side of the equation is $\sum_{j=1}^{n}\langle \prod_{i=1}^{n}\psi_i^{d_i-\delta_{ij}} \rangle^m.$ This is just the proof of the original string equation, but we recall the argument for the convenience of the reader. We have $(q^*\psi_1)^{d_1}(q^*\psi_2)^{d_2} \cdots (q^*\psi_n)^{d_n}=0$ since $\dim \SM_{1,n}(m)=n$. Next, since the divisors $\{\Delta_{\{i,n+1\}}\}_{i=1}^{n}$ are disjoint, we can expand the lefthand side as
\begin{align*}
\sum_{i=1}^{n}\left( \sum_{j=1}^{d_i}{d_i \choose j} (q^*\psi_i)^{d_i-j}\Delta_{\{i,n+1\}}^{j} \right) \prod_{k \neq i}(q^*\psi_k)^{d_k}.
\end{align*}

We can evaluate this sum as an intersection product on $\SM_{1,n}(m)$ by using the natural identifications:
\begin{align*}
\Delta_{\{i,n+1\}} &\simeq \SM_{1,n}(m),\\
\Delta_{\{i,n+1\}}|_{\Delta_{\{i,n+1\}}} &\simeq -\psi_i,\\
q^*\psi_j|_{\Delta_{\{i,n+1\}}} &\simeq \psi_j.\\
\end{align*}
Using the fact that $\sum_{j=1}^{d_i}(-1)^{j-1}{d_i \choose j}=1$, we obtain
\begin{align*}
\deg LHS&=\deg \sum_{i=1}^{n}\left( \sum_{j=1}^{d_i}{d_i \choose j} (q^*\psi_i)^{d_i-j}|_{\Delta_{\{i,n+1\}}} \Delta_{\{i,n+1\}}^{j-1}|_{\Delta_{\{i,n+1\}}}  \right) \prod_{k \neq i}(q^*\psi_k)^{d_k}|_{\Delta_{\{i,n+1\}}} .\\
&=\deg_{\SM_{1,n}(m)}  \sum_{i=1}^{n}\left( \sum_{j=1}^{d_i}(-1)^{j-1}{d_i \choose j} \psi_i^{d_i-1} \right) \prod_{k \neq i}\psi_k^{d_k}\\
&=\deg_{\SM_{1,n}(m)}   \sum_{i=1}^{n} (\psi_i)^{d_i-1} \prod_{k \neq i} \psi_k^{d_k}.\\
&=\sum_{j=1}^{n} \langle \prod_{i=1}^{n}\psi_i^{d_i-\delta_{ij}} \rangle^m
\end{align*}

Next, we evaluate the degree of the righthand side. By the push-pull formula, $$\deg \,(p^*\psi_1)^{d_1}(p^*\psi_2)^{d_2} \cdots (p^*\psi_n)^{d_n} = \langle \psi_1^{d_1} \ldots \psi_n^{d_n} \rangle^m.$$ Since the exceptional divisors are disjoint, we can then write the degree of the righthand side as
\begin{align*}
\deg RHS =\langle \psi_1^{d_1}\ldots \psi_n^{d_n} \rangle^{m}+\sum_{S \in {n \brack m}} \deg Z_{S},
\end{align*}
where $Z_{S} \in A^*(E_S)$ is the class determined by the sum of all terms divisible by $E_{S}$ (and only by $E_{S}$). Explicitly, if $S=\{S_1, \ldots, S_k, \{i_1\}, \ldots, \{i_{m-k}\}\},$
then
$$
Z_{S}=\prod_{i \in S_1 \cup \ldots \cup S_k}(p^*\psi_i)^{d_i} \cdot \frac{(p^*\psi_{i_1}+E_S)^{d_{i_1}} \cdots (p^*\psi_{i_{m-k}}+E_S)^{d_{i_{m-k}}}-(p^*\psi_{i_1})^{d_{i_1}}\cdots(p^*\psi_{i_{m-k}})^{d_{i_{m-k}}}}{E_{S}}\Bigg|_{E_S}.
$$

To complete the proof Theorem $1^*$(a), it now suffices to show that
$$
\deg Z_{S}=(-1)^{\star(S)+1}\,\langle\psi_1^{m+1} \rangle^m\, \prod_{j=1}^{k}{|S_j|-1 \choose \{d_i\}_{i \in S_j}}.
$$

We claim that this is precisely the content of Proposition \ref{P:intersection}(b) in Section \ref{S:intersection}. To see this, first observe that that $E_{S}$ is isomorphic to the projective bundle $Y$ appearing in the statement of Proposition \ref{P:intersection}. Indeed, we have
\begin{align*}
\Delta_{S}:&=\Delta_{0,S_1} \cap \ldots \cap \Delta_{0,S_k}\\
 &\simeq \SM_{1,m+1}(m) \times \SM_{0,|S_1|+1} \times \ldots \times \SM_{0,|S_k|+1},\\
E_{S}:&=\P(N),
\end{align*}
where $N$ is the normal bundle of $\Delta_{S} \subset \SM_{1,n}(m)$. Since $\Delta_{S}$ is a global complete intersection, we have 
$$
N=\oplus_{i=1}^{k} \O(\Delta_{0,S_i})|_{\Delta_S},
$$
and by the standard identification of the deformation space of a node with the tensor product of the tangent spaces of its branches, $\O(\Delta_{0,S_i})|_{\Delta_S}=T_i \oplus T_i',$ where $T_i$ (resp. $T_i'$) is the pull-back of the tangent bundle of the $i^{th}$ section over $\SM_{1,m+1}(m)$ (resp. $(|S_i|+1)^{st}$ section over $\SM_{0,|S_i|+1}$). Thus, $N$ is precisely the bundle appearing in the definition of $Y$ in Section \ref{S:intersection}.

In terms of the presentation of $A^*(Y)$ described Section \ref{S:intersection}, the classes appearing in the definition of $Z_{S}$ are simply
\begin{align*}
E_{S}|_{E_S}&=\eta, \\
p^*\psi_{i_j}|_{E_S}&=x_0, \text{ $j=1,2 \ldots, m-k,$}\\
p^*\psi_{j}|_{E_S}&=\psi_j, \text{ $j \in S_1 \cup \ldots \cup S_k.$}\\
\end{align*}
Thus, we have
\begin{align*}
Z_{S}&= \frac{(p^*\psi_{i_1}+E_S)^{d_{i_1}} \cdots (p^*\psi_{i_{m-k}}+E_S)^{d_{i_{m-k}}}-(p^*\psi_{i_1})^{d_{i_1}}\cdots(p^*\psi_{i_{m-k}})^{d_{i_{m-k}}}}{E_{S}}\Big|_{E_S}\cdot \prod_{i \in S_1 \cup \ldots \cup S_k}(p^*\psi_i)^{d_i}\Big|_{E_S} \\
&=\frac{(x_0+\eta)^{d_{i_1}} \cdots (x_0+\eta)^{d_{i_{m-k}}}-(x_0)^{d_{i_1}}\cdots(x_0)^{d_{i_{m-k}}}}{\eta}\, \cdot \prod_{i \in S_1 \cup \ldots \cup S_k}\psi_i^{d_i}  \in A^{n}(Y)\\
&=\frac{(x_0+\eta)^d-x_0^d}{\eta}\, \cdot \prod_{i \in S_1 \cup \ldots \cup S_k}\psi_i^{d_i}  \in A^{n}(Y),\\
\end{align*}
where $d=d_{i_1}+\ldots+d_{i_{m-k}}=(n+1)-\sum_{i \in S_1 \cup \ldots \cup S_k}d_i$. Now Proposition \ref{P:intersection}(b) asserts
$$
\deg Z_{S}=(-1)^{\star(S)+1}\,\langle\psi_1^{m+1} \rangle^m\, \prod_{j=1}^{k}{|S_j|-1 \choose \{d_i\}_{i \in S_j}},
$$
as desired.\\

\subsubsection{$m$-stable Dilaton equation}\label{S:dilaton}

To prove Theorem $1^*\,$(b), we again consider the resolution of the forgetful map from Section \ref{S:forgetful}:

\[
\xymatrix{ 
&\X  \ar[ld]_{p} \ar[rd]^{q} &\\
\SM_{1,n+1}(m) \ar@{-->}[rr]&&\SM_{1,n}(m)\\
}
\]
Given non-negative integers $d_1, \ldots, d_n$ satisfying $\sum_{i=1}^{n} d_i=n$, Corollary \ref{C:psi-reduction} gives the following equation in $A_{\Q}^{n+1}(\X)$:
\begin{align*}
(q^*\psi_1+\Delta_{\{1,n+1\}})^{d_1}\cdots (q^*\psi_n+\Delta_{\{n,n+1\}})^{d_n} \cdot p^*\psi_{n+1}=(p^*\psi_1+\sum_{\substack{S \in {n \brack m}\\ \{1\} \in S }}E_{S})^{d_1}\cdots (p^*\psi_n+\sum_{\substack{S \in {n \brack m}\\ \{n\} \in S }}E_{S})^{d_n} \cdot p^*\psi_{n+1}.\\
\end{align*}

First, we show that the degree of the lefthand side is $n \langle \psi_1^{d_1} \ldots \psi_n^{d_n} \rangle$.  This is just the proof of the original dilaton equation, but we recall the argument for the convenience of the reader. First, observe that since $p^{*}\psi_{n+1}|_{\Delta_{\{i,n+1\}}}=0,$ all terms on the lefthand side are zero except the leading term. To see that
$$
 \deg \,(q^*\psi_1)^{d_1}(q^*\psi_2)^{d_2} \cdots (q^*\psi_n)^{d_n} \cdot p^*\psi_{n+1} = n \langle \psi_1^{d_1} \ldots \psi_n^{d_n} \rangle^m,
$$
it suffices to see that $q_*\left(p^*\psi_{n+1}\right)=n\left[\,\SM_{1,n}(m)\right]$. This follows by a standard test-curve argument. Let $T \subset \SM_{1,n+1}(m)$ be the curve obtained by taking a fixed $n$-pointed, smooth elliptic curve, and letting the $(n+1)^{st}$ marked point vary along the curve (and blowing up when the $(n+1)^{st}$ point collides with the other marked points). Then $T$ is a contracted curve which avoids the indeterminacy locus of $\SM_{1,n+1}(m) \dashrightarrow \SM_{1,n}(m)$, and $\deg \psi_{n+1}|_{T}=n$.

Next, we evaluate the degree of the righthand side. By the push-pull formula, the degree of the leading term is precisely $\langle \psi_1^{d_1} \ldots \psi_n^{d_n} \psi_{n+1} \rangle^{m}. $ Since the exceptional divisors are disjoint, we can then write the degree of the the righthand side as
\begin{align*}
\langle \psi_1^{d_1}\ldots \psi_n^{d_n} \rangle^{m}+\sum_{S \in {n \brack m}} \deg Z_{S}.
\end{align*}
where $Z_{S} \in A^*(E_S)$ is the class determined by the sum of all terms divisible by $E_{S}$ (and only by $E_{S}$). Arguing precisely as in the proof of the string equation (\ref{S:string} above), we see that $E_{S}$ is isomorphic to the projective bundle $Y$ defined in Section \ref{S:intersection}, and that in terms of the presentation of $A^*(Y)$ given there, we have
$$
Z_{S}=x_0 \cdot \frac{(x_0+\eta)^d-x_0^d}{\eta} \cdot \prod_{i \in {S_1 \cup \ldots \cup S_k}}\psi_i^{d_i} \in A^{n}(Y),
$$
where $d=n-\sum_{i \in S_{1} \cup \ldots \cup S_k}d_i$.
Thus, Proposition \ref{P:intersection}(c) says
$$
\deg Z_{S}=(-1)^{\star(S)+1}\,\langle \psi_1^{m+1} \rangle^m \, \prod_{j=1}^{k}{|S_j|-1 \choose \{d_i\}_{i \in S_j}}.
$$
Equating degrees of lefthand and righthand sides gives
$$
\langle \psi_1^{d_1}\ldots \psi_n^{d_n} \psi_{n+1} \rangle^{m}= n \langle \psi_1^{d_1} \ldots \psi_n^{d_n} \rangle^m+\sum_{S \in {n \brack m}}(-1)^{\star(S)}\,\langle\psi_1^{m+1} \rangle^m\, \prod_{j=1}^{k}{|S_j|-1 \choose \{d_i\}_{i \in S_j}},
$$
as desired.

\subsubsection{Reduction recursion}
To prove Theorem $2^*$, we consider the resolution of the reduction map from Section \ref{S:reduction}:

\[
\xymatrix{ 
&\X  \ar[ld]_{p} \ar[rd]^{q} &\\
\SM_{1,n}(m) \ar@{-->}[rr]&&\SM_{1,n}(m+1)\\
}
\]
Given non-negative integers $d_1, \ldots, d_n$ satisfying $\sum_i d_i=n$, Corollary \ref{C:psi-reduction} gives the following equality in $A_{\Q}^{n}(\X)$:
\begin{align*}
(q^*\psi_1)^{d_1}(q^*\psi_2)^{d_2} \cdots (q^*\psi_n)^{d_n}=(p^*\psi_1+\sum_{\substack{S \in {n \brack m+1}\\ \{1\} \in S }}E_{S})^{d_1}\cdots (p^*\psi_n+\sum_{\substack{S \in {n \brack m+1}\\ \{n\} \in S }}E_{S})^{d_n}\\
\end{align*}

The degree of the left side of this equation is $\langle \psi_1^{d_1} \ldots \psi_n^{d_n} \rangle^{m+1},$ and the degree of the leading term of the right side is $\langle \psi_1^{d_1} \ldots \psi_n^{d_n} \rangle^{m}.$  Since the exceptional divisors are disjoint, we then have
\begin{align*}
\langle \psi_1^{d_1} \ldots \psi_n^{d_n} \rangle^{m+1}=\langle \psi_1^{d_1}\ldots \psi_n^{d_n} \rangle^{m}+\sum_{S \in {n \brack m+1}}\deg Z_{S},
\end{align*}
where $Z_{S} \in A^*(E_S)$ is the class determined by the sum of all terms divisible by $E_{S}$ (and only by $E_{S}$).

Just as in the proofs of the string/dilaton equations (\ref{S:string} and \ref{S:dilaton} above), $E_{S}$ is isomorphic to the projective bundle $Y$ of Section 3.2, and in terms of the presentation of $A^*(Y)$ given there, we have
\begin{align*}
Z_{S}:&=\frac{(x_0+\eta)^{d}-x_0^d}{\eta} \prod_{i \in S_1 \cup \ldots \cup S_k}\psi_i^{d_i} \in A^{n-1}(Y),\\
\end{align*}
where $d=n-\sum_{i \in S_1 \cup \ldots \cup S_k} d_i$. By Proposition \ref{P:intersection} (a), we have 
$$
\deg Z_{S}=(-1)^{\star(S)}\,\langle \psi_1^{m+1} \rangle^{m}\, \prod_{j=1}^{k}{|S_j|-1 \choose \{d_i\}_{i \in S_j}},
$$
and the result follows.

\subsection{Key Intersection Theory Calculation}\label{S:intersection}
Let $S_1, \ldots, S_k$ be nonempty, disjoint subsets of $[n]$ satisfying $|S_i| \geq 2$, and consider the stack
$$
X:=\SM_{1,m+1}(m) \times \SM_{0,|S_1|+1} \times  \SM_{0,|S_2|+1} \times \ldots \times \SM_{0,|S_k|+1}.
$$
We consider the first $|S_i|$ sections of  $\SM_{0,|S_i|+1}$ as labeled by the elements of $S_i$, and define $\{\psi_j \in A^1(\SM_{0,|S_i|+1}): j \in S_i\}$ as the chern classes of the corresponding cotangent bundles. For each $i \in \{1, \ldots, k\},$ we define $x_i \in A^1(\SM_{0,|S_i|+1})$ as the chern class of the cotangent bundle of the $(|S_i|+1)^{st}$ section. Finally, we define $x_0 \in A^1(\SM_{1,m+1}(m))$ as the chern class of the cotangent bundle of any of the $m+1$ sections over $\SM_{1,m+1}(m)$ (the cotangent bundles of the different sections are all linearly equivalent by Proposition 3.2 in \cite{smyth_elliptic2}).  We also consider $x_0, x_1, \ldots, x_k, \{\psi_{j} : j \in S_1 \cup \ldots \cup S_k\}$ as classes in $A^1(X)$ via pullback.

For each $i \in \{1, \ldots, k\}$, let $T_i \in \Pic(X)$ denote the pullback of the tangent bundle of the $i^{th}$-section over $\SM_{1,m+1}(m)$, and let $T_i' \in \Pic(X)$ denote the pullback of the tangent bundle of the $(|S_i|+1)^{st}$ section over $\SM_{0,|S_i|+1}.$ Set
$$
N:=\oplus_{i=1}^{k}T_i \oplus T_i',
$$
and observe that the chern classes of $N$ are given by
\begin{align*}
c_i(N)&=s_i(-x_0-x_1, -x_0-x_2, \ldots, -x_0-x_k),\\
&=(-1)^{i}s_i(x_0+x_1, x_0+x_2, \ldots, x_0+x_k).
\end{align*}
where $s_i$ is the $i^{th}$ elementary symmetric function in $k$ variables.

Now we define $Y:=\P(N)$ to be the projectivization of $N$ over $X$. We have
\begin{align*}
A^*(Y):&=A^*(X)[\eta]/(\eta^k-c_1(N)\eta^{k-1}+c_2(N)\eta^{k-2}-\ldots+(-1)^k c_k(N)),\\
&=A^*(X)[\eta]/(\eta^k+s_1\eta^{k-1}+s_2\eta^{k-2}+\ldots+s_k),
\end{align*}
where $\eta=c_1(\O_{\P}(-1)) \in A^1(Y)$, and $s_i:=s_i(x_0+x_1, x_0+x_2, \ldots, x_0+x_k) \in A^i(X)$. 

The proofs of Theorems $1^*$(a), $1^*$(b), and $2^*$ each require computing the degree of a certain class on $Y$. The necessary results are stated in the following proposition.
\begin{proposition}\label{P:intersection}
\begin{enumerate}
\item[]
\item[(a)] Suppose that $\sum_{i=1}^{k}|S_i|=n-m+k-1$, so $\dim Y=n-1$. 
Let $d_1, \ldots, d_n$ be a collection of non-negative integers such that $\sum_{i=1}^{n} d_i=n$, and let 
$$d:=\sum_{i \in [n] / S_1 \cup \ldots \cup S_k} d_i=n-\sum_{i \in S_1 \cup \ldots \cup S_k}d_i.$$
Let $Z$ be the following class on $Y$:
$$Z:=\frac{(x_0+\eta)^d-x_0^d}{\eta} \cdot \prod_{i \in {S_1 \cup \ldots \cup S_k}}\psi_i^{d_i} \in A^{n-1}(Y).$$
Then $$\deg Z=(-1)^{\star} \langle \psi_1^{m+1} \rangle^m  \prod_{j=1}^{k}{|S_j|-1 \choose \{d_i\}_{i \in S_j}}.$$
where $\star=n-m-k-1-\sum_{j \in S_1 \cup \ldots \cup S_k} d_j$.\\

\item[(b)] Suppose that $\sum_{i=1}^{k}|S_i|=n-m+k$, so $\dim Y=n$. 
Let $d_1, \ldots, d_n$ be a collection of non-negative integers such that $\sum_{i=1}^{n} d_i=n+1$, and let 
$$d:=\sum_{i \in [n] / S_1 \cup \ldots \cup S_k} d_i = (n+1)-\sum_{i \in S_1 \cup \ldots \cup S_k}d_i.$$
Let $Z$ be the following class on $Y$:
$$Z:=\frac{(x_0+\eta)^d-x_0^d}{\eta} \cdot \prod_{i \in {S_1 \cup \ldots \cup S_k}}\psi_i^{d_i} \in A^{n}(Y).$$
Then $$\deg Z=(-1)^{\star+1} \langle \psi_1^{m+1} \rangle^m  \prod_{j=1}^{k}{|S_j|-1 \choose \{d_i\}_{i \in S_j}},$$
where $\star=n-m-k-1-\sum_{j \in S_1 \cup \ldots \cup S_k} d_j$.\\

\item[(c)] Suppose that $\sum_{i=1}^{k}|S_i|=n-m+k$, so $\dim Y=n$. 
Let $d_1, \ldots, d_n$ be a collection of non-negative integers such that $\sum_{i=1}^{n} d_i=n$, and let 
$$d:=\sum_{i \in [n] / S_1 \cup \ldots \cup S_k} d_i=n-\sum_{i \in S_1 \cup \ldots \cup S_k}d_i.$$
Let $Z$ be the following class on $Y$:
$$Z:=x_0 \cdot \frac{(x_0+\eta)^d-x_0^d}{\eta} \cdot \prod_{i \in {S_1 \cup \ldots \cup S_k}}\psi_i^{d_i} \in A^{n}(Y),$$
Then $$\deg Z=(-1)^{\star+1} \langle \psi_1^{m+1} \rangle^m  \prod_{j=1}^{k}{|S_j|-1 \choose \{d_i\}_{i \in S_j}},$$
where $\star=n-m-k-1-\sum_{j \in S_1 \cup \ldots \cup S_k} d_j$.\\

\end{enumerate}
\end{proposition}
\begin{proof}
The proofs of (a), (b), and (c) are essentially identical, so we just prove (a). The idea is to use the fundamental relation in $A^{*}(Y)$ to rewrite $((x_0+\eta)^d-x_0^d)/\eta$ as a polynomial of degree $(k-1)$ in $\eta$, i.e. to write
  $$
\frac{(x_0+\eta)^d-x_0^d}{\eta} =q_{0}\eta^{k-1}+q_1\eta^{k-2}+\ldots +q_{k-1},
$$
for some $q_i \in A^{*}(X)$. Evidently, the classes $q_i$ will be polynomials in $x_0, x_1, \ldots, x_k,$ and we will have
$$
\deg Z=
\deg q_0(x_0, \ldots, x_k)\prod_{i \in {S_1 \cup \ldots \cup S_k}}\psi_i^{d_i} \in A^{n-k}(X).\\
$$

Furthermore, since $X \simeq \SM_{m+1}(m) \times \SM_{0,|S_1|+1} \times \ldots \times \SM_{0,|S_k|+1},$ a monomial of the form
$$
x_0^{c_0}x_1^{c_1} \ldots x_k^{c_k}\prod_{i \in S_1 \cup \ldots \cup S_k}\psi_i^{d_i}
$$
can only give a nonzero class in $A^{n-k}(X)$ if it contains precisely $m+1$ classes pulled back from $\SM_{1,m+1}(m)$, and $|S_i|-2$ classes pulled back from $\SM_{0,|S_i|+1}$ (for each $i=1, \ldots, k$). In other words, if we define the \emph{deficiencies} to be the integers
$$
e_i=|S_i|-2-\sum_{j \in S_i} d_j,\,\,\,\, i=1, \ldots, k,
$$
then this monomial contributes to the degree of $Z$ only if $c_0=m+1$ and $c_i=e_i$ for $i=1, \ldots, k$. Finally, if these equalities do hold, then the formula \footnote{This closed formula for Witten-Kontsevich numbers on $\M_{0,n}$ is an elementary consequence of the string equation.}
\begin{align*}
\langle \psi_1^{d_1} \ldots \psi_n^{d_n} \rangle_{0,n}&= {n-3 \choose d_1, \ldots, d_n }\\ 
\end{align*}
implies that
$$ 
\deg x_0^{m+1}x_1^{e_1}\ldots x_k^{e_k}  \cdot \prod_{i \in {S_1 \cup \ldots \cup S_k}}\psi_i^{d_i} =\langle \psi_1^{m+1} \rangle^m \prod_{j=1}^{k}{|S_j|-2 \choose \{d_i\}_{i \in S_j}}.
$$
Thus, to prove statement (a) of the Proposition, it only remains to show that when we reduce $((x_0+\eta)^d-x_0^d)/\eta$ to a polynomial of degree $(k-1)$ in $\eta$, the coefficient of $\eta^{k-1}$ contains the monomial $x_0^{m+1}x_1^{e_1}\ldots x_k^{e_k}$ with coefficient precisely $(-1)^{\star}$.

To do this, we use three combinatorial lemmas (proved below). Lemma \ref{L:reduce1} implies that when we reduce
$$\eta^{d-1}+ {d \choose 1}x_0\eta^{d-2}+{d \choose 2}x_0^2\eta^{d-3}+\ldots+{d \choose d-1}x_0^{d-1}$$
 the coefficient of $\eta^{k-1}$ is exactly $(-1)^{d-k}$ times the alternating sum
\begin{multline*}
p_{d-k}(x_0+x_1, \ldots, x_0+x_k)-{ d \choose 1}x_0p_{d-k-1}(x_0+x_1, \ldots, x_0+x_k)+\ldots\\-{d \choose d-k-1} x_0^{d-k-1}p_{1}(x_0+x_1, \ldots, x_0+x_k) + {d \choose d-k}x_0^{d-k}.
\end{multline*}

Lemma \ref{L:reduce2} implies that the coefficient of $x_0^{m+1}x_1^{e_1}\ldots x_k^{e_k}$ in the term-by-term expansion of this polynomial is $(-1)^{d-k}$ times the alternating sum
$$
{d-1 \choose m+1}-{ d \choose 1}{d-2 \choose m}+{ d \choose 2}{d-3 \choose m-1}-\ldots \pm {d \choose m+1}{d-m-2 \choose 0}.
$$
Finally, Lemma \ref{L:reduce3} shows that this alternating sum of binomial coefficients is just $(-1)^{m+1}$. Thus, we find that  $x_0^{m+1}x_1^{e_1}\ldots x_k^{e_k}$ appears with coefficient $(-1)^{d-k} \cdot (-1)^{m+1}=(-1)^{d-k-m-1}=(-1)^{\star}$ as desired. \\

\end{proof}

\begin{lemma}\label{L:reduce1}
When we expand $\eta^{d+k-1} \in A^{*}(Y)$ in terms of the basis $\eta^{k-1}, \eta^{k-2}, \ldots, \eta, 1$, the coefficient of $\eta^{k-1}$ is precisely $(-1)^{d}p_d(x_0+x_1, \ldots, x_0+x_k),$ where $p_d$ is the sum of all degree $d$ monomials in $x_1, \ldots, x_k$, i.e. $$p_d(x_1, \ldots, x_k):=\sum_{1 \leq i_1 \leq i_2 \leq  \ldots  \leq i_d \leq k}^{k}x_{i_1}x_{i_2}\ldots x_{i_d},$$
\end{lemma}
\begin{proof}
We prove a slightly more general statement. Define symmetric polynomials $q_{d,i} \in \mathbb{C}[x_0, x_1, \ldots, x_k]$ by the formula
$$
\eta^{d+(k-1)}=q_{d,0}\eta^{k-1}+q_{d,1}\eta^{k-2}+\ldots+q_{d,k-1}.
$$
The definition implies that the polynomials $q_{d,i}$ satisfy the following initial condition and recursion:
\begin{align*}
q_{1,i}&=-s_{i+1},\\
q_{d,i}&=q_{d-1,i+1}-s_{i+1}q_{d-1,0},
\end{align*}
where $s_i:=s_i(x_0+x_1, \ldots, x_0+x_k)$ as in our discussion of $A^*(Y)$.

We will prove by induction on $d$ that this recursion is solved by the following formula:
$$
(-1)^dq_{d,i}=s_{i+1}p_{d-1}-s_{i+2}p_{d-2}+\ldots+(-1)^{d-1}s_{d+i},
$$
where $p_i:=p_i(x_0+x_1, \ldots, x_0+x_k)$. Note that we have the following basic combinatorial identity (the inclusion-exclusion principle):
$$
p_{d}-s_1p_{d-1}+s_2p_{d-1}-\ldots+(-1)^{d}s_d=0,
$$
so this will show in particular that $q_{d,0}=(-1)^dp_d$ as required.

Assuming that the claim is true for $d-1$, we have
\begin{align*}
(-1)^{d-1}q_{d-1,0}&=p_{d-1}\\
(-1)^{d-1}q_{d-1,i+1}&=s_{i+2}p_{d-2}-\ldots+(-1)^{d-2}s_{d+i}
\end{align*}

Thus, applying the recursion gives
\begin{align*}
(-1)^dq_{d,i}&=(-1)^d(q_{d-1,i+1}-s_{i+1}q_{d-1,0})\\
&=(-1)^{d-1}s_{i+1}q_{d-1,0}-(-1)^{d-1}q_{d-1,i+1}\\
&=s_{i+1}p_{d-1}-s_{i+2}p_{d-2}+\ldots+(-1)^{d-1}s_{d+i},
\end{align*}
as desired.

\end{proof}

\begin{lemma}\label{L:reduce2}
The coefficient of $x_0^{m}x_1^{e_1}\ldots x_k^{e_k}$ in the term-by-term expansion of $p_{m+\sum_{i=1}^{k}e_i}(x_0+x_1, \ldots, x_0+x_k)$ is ${m+\sum_ie_i+k-1 \choose m}$.
\end{lemma}
\begin{proof}
For any choice of nonnegative integers $f_1, \ldots, f_k$ such that $f_1+\ldots+f_k=m$, the coefficient of $x_0^{m}x_1^{e_1}\ldots x_k^{e_k}$ in the expansion of 
$$\prod_{i=1}^{k}(x_0+x_i)^{e_i+f_i}$$
is given by $\prod_{i=1}^{k}{e_i+f_i \choose e_i}.$ It follows that the coefficient of $x_0^{m}x_1^{e_1}\ldots x_k^{e_k}$ in $p_{m+\sum_{i=1}^{k}e_i}(x_0+x_1, \ldots, x_0+x_k)$ is 
$$\sum_{f_1+\ldots+f_k=m} {e_1+f_1 \choose e_1}{e_2+f_2 \choose e_2} \cdots {e_k+f_k \choose e_k},$$
where the sum is taken over all partitions of $m$ into nonnegative integers $f_1, \ldots, f_k$. Thus, it suffices to establish the identity
$$\sum_{f_1+\ldots+f_k=m} {e_1+f_1 \choose e_1}{e_2+f_2 \choose e_2} \cdots {e_k+f_k \choose e_k}={m+\sum_{i=1}^{k}e_i+k-1 \choose \sum_{i=1}^{k}e_i+k-1}.$$

Consider a row of $\sum_{i=1}^k e_i+m+k-1$ marbles, with the first $e_1$ marbles having color 1, the next $e_2$ marbles having color 2, etc., and the last $m+k-1$ marbles having color $k+1$, which we might as well call black. The righthand side of our identity counts subsets of this row of marbles of size $\sum_{i=1}^ke_i+k-1$. We will show that the lefthand side counts the same thing. Given a partition $f_1+\ldots+f_k=m,$ we can divide the black marbles into $k$ sections of lengths $f_1, \ldots, f_k$ separated by $k-1$ walls, i.e. designate the $(f_1+1)^{st}, (f_1+f_2+2)^{nd}, \ldots, (f_{1}+\ldots+f_{k-1}+k-1)^{st}$ black marbles as walls. We can then pick a subset of size $\sum_{i=1}^{k}e_i+k-1$ by declaring the $k-1$ walls to be in the subset, and additionally taking exactly $e_i$ marbles which are either of color $i$ or black in section $i$ (for each $i=1,\ldots, k$). This gives $\prod_{i=1}^{k}{e_i+f_i \choose e_i}$ distinct subsets using the designated walls. As we range over all possible partitions of $m$ (i.e. all choices of walls), we choose each subset of size $\sum_{i=1}^ke_i+k-1$ exactly once.

\end{proof}

\begin{lemma}\label{L:reduce3}
For any integers $1 \leq m < d,$
we have
$$
{d-1 \choose m}-d{d-2 \choose m-1}+{d \choose 2}{d-3 \choose m-2}+\ldots \pm {d \choose m}{d-m-1 \choose 0}=(-1)^{m}
$$
\end{lemma}

\begin{proof}
Define $$q(a,b,m)= {a \choose 0} {b \choose m} -{ a \choose 1} {b-1 \choose m-1} + {a \choose 2} {b-2 \choose m-2} - \ldots + (-1)^{m} {a \choose m} {b-m \choose 0},$$ for any nonnegative integers $a, b, m$ satisfying $a, b \geq m \geq 1$. We wish to show that $q(d,d-1,m)=(-1)^{m}$.

These sums are easily seen to satisfy the following two identities.
\begin{enumerate}
\item $q(a,a,m)=0$.
\item $q(a,a-1, m)=q(a-1,a-1,m)-q(a-1,a-2,m-1)$.
\end{enumerate}
The second identity is an easy consequence of Pascal's formula. For the first identity, simply observe that
\begin{align*}
q(a,a,m)&= {a \choose 0} {a \choose m} -{ a \choose 1} {a-1 \choose m-1} + {a \choose 2} {a-2 \choose m-2} - \ldots + (-1)^{m} {a \choose m} {a-m \choose 0}\\
&={m \choose 0} {a \choose m} -{ m \choose 1} {a \choose m} + {m \choose 2} {a \choose m} - \ldots + (-1)^{m} {m \choose m} {a \choose m}\\
&=0.
\end{align*}
From these two identities, the fact that $q(d,d-1,m)=(-1)^m$ follows immediately by induction.

\end{proof}

\subsection{Proof of $m$-stable Initial Condition}
In this section, we prove Theorem 3, which states that every Witten-Kontsevich number on $\M_{1,m+1}(m)$ is equal to $m!/24$. In fact, it suffices to prove that
$$
\langle \psi_1^{m+1} \rangle^{m}=\frac{m!}{24}.
$$
Indeed, since the $\Q$-Picard group of $\M_{1,m+1}(m)$ has rank one, all the $\psi_i$-classes are all equal (Proposition 3.2 in \cite{smyth_elliptic2}), hence all Witten-Kontsevich numbers on $\M_{1,m+1}(m)$ have the same value. We will use Theorems $1^*$ and $2^*$ to evaluate $\langle \psi_1^{m+1} \rangle^{m}$ inductively.

First, we apply Theorem $1^*$(a) to $\M_{1,m+1}(m-1) \dashrightarrow \M_{1,m}(m-1)$ to get
\begin{align*}
\langle \psi_1^{m+1} \rangle^{m-1}&=\langle \psi_1^m \rangle^{m-1}+ \langle \psi_1^{m} \rangle^{m-1} \sum_{S \in {m \brack m-1}}(-1)^{\star(S)}\prod_{j=1}^{k(S)} {|S_j|-2\choose \sum_{i \in S_j}d_i}.
\end{align*}
Note that if $S=\{S_1, \ldots, S_{m-1}\}$ is an $(m-1)$-partition of $[m]$, then we must have $k(S)=1$ and $|S_1|=2$. Furthermore, if $1 \in S_1,$ then we have

$${|S_1|-2 \choose \sum_{i \in S_1}d_i}=0,$$
since $d_1> |S_1|-2=0$. Thus, the only partitions that give rise to nonzero error terms are those which additionally satisfy $1 \notin S_1$. There are ${m-1\choose 2}$ such partitions, and for each of them, we have
\begin{align*}
\prod_{j=1}^{k(S)} {|S_j|-2\choose \sum_{i \in S_j}d_i}&={|S_1|-2 \choose \sum_{i \in S_1}d_i}= {0 \choose 0}=1.\\
\star(S):&=n-m-k(S)-\sum_{j=1}^{k(S)}d_j-1\\
&=m-(m-1)-1-0-1\\
&=-1.
\end{align*}
Thus, we get the formula
\[
\langle \psi_1^{m+1} \rangle^{m-1}=\left[1- {m-1 \choose 2}\right] \langle \psi_1^m \rangle^{m-1}.\\ 
\eqno{(\dagger)}
\]\\

Next, we apply Theorem $2^*$ to $\M_{1,m+1}(m-1) \dashrightarrow \M_{1,m+1}(m)$ to obtain
\begin{align*}
\langle \psi_1^{m+1} \rangle^{m}&=\langle \psi_1^{m+1} \rangle^{m-1}+ \langle \psi_1^{m} \rangle^{m-1} \sum_{S \in {m+1 \brack m}}(-1)^{\star(S)}\prod_{j=1}^{k(S)} {|S_j|-2\choose \sum_{i \in S_j}d_i}.
\\ 
\end{align*}
Note that if $S=\{S_1, \ldots, S_{m}\}$ is an $m$-partition of $[m+1]$, then we must have $k(S)=1$ and $|S_1|=2$. Furthermore, if $1 \in S_1,$ then we have
$${|S_1|-2 \choose \sum_{i \in S_1} d_i}=0,$$
since $d_1> |S_1|-2=0$. Thus, the only partitions that give rise to nonzero error terms are those which additionally satisfy $1 \notin S_1$. There are ${m \choose 2}$ such partitions, and for each of them, we have
\begin{align*}
\prod_{j=1}^{k(S)} {|S_j|-2\choose \sum_{i \in S_j}d_i}&={|S_1|-2 \choose \sum_{i \in S_1}d_i}= {0 \choose 0}=1.\\
\star(S):&=n-m-k(S)-\sum_{j=1}^{k(S)}d_j-1\\
&=(m+1)-(m-1)-1-0-1\\
&=0.\\
\end{align*}

Thus, we get the formula 
\[
\langle \psi_1^{m+1} \rangle^{m}=\langle \psi_1^{m+1} \rangle^{m-1}+{m \choose 2}\langle \psi_1^m \rangle^{m-1}.\\ 
\eqno{(\dagger\dagger)}
\]

Combining $(\dagger)$ and $(\dagger\dagger)$, we obtain
\begin{align*}
\langle \psi_1^{m+1} \rangle^{m}&=\left[{m \choose 2}- {m-1 \choose 2}+1\right] \langle \psi_1^m \rangle^{m-1}=m\,  \langle \psi_1^m \rangle^{m-1}.\\
\end{align*}

The formula $\langle \psi_1^{m+1} \rangle^{m}=m!/24$ follows immediately by induction on $m$ (using the well-known base case $\langle \psi_1 \rangle^0=\deg_{\M_{1,1}}\psi_1=1/24$).

\section{Sample Calculations}
In this section, we explain how to compute the Witten-Kontsevich numbers appearing in Table 1.
In the table, we use Witten's $\tau$-notation \cite{Witten}, setting
$$
\langle \prod_{j \in \N} \tau_j^{d_j} \rangle^m:=\langle \psi_1^{k_1} \psi_2^{k_2} \ldots \psi_n^{k_n} \rangle^m
$$
where $\{\tau_j: j \in \N\}$ are viewed as commuting formal variables, and
$$d_j:=\# \{i \,|\, k_i=j \}.$$
 This notation avoids redundancy in the labeling of Witten-Kontsevich numbers, e.g. $\langle \tau_0^2\tau_3\rangle$ gives a single label for the three obviously equal numbers
$$
\langle\psi_1^0 \psi_2^0 \psi_3^3 \rangle=\langle\psi_1^0 \psi_2^3 \psi_3^0 \rangle=\langle\psi_1^3 \psi_2^0 \psi_3^0 \rangle.
$$
In Table 1, we have suppressed the $m$-superscript, since it is implicit in the position of the entries. Also, in order to keep the entries of the table integral (and therefore easier to read), we have opted to renormalize the Witten-Kontsevich numbers by multiplying by a factor of 24. (This is tantamount to setting $\langle \tau_1 \rangle=1$, instead of its true value of 1/24.) 
 
The entire table can be built up from the initial entry $\langle \tau_1 \rangle=1$ by repeatedly applying Theorems 1 and 2.  As an illustration, we compute all Witten-Kontsevich numbers on $\M_{1,4}(2)$ (given that all Witten-Kontsevich numbers on $\M_{1,3}(2)$ and $\M_{1,4}(1)$ are known). First, however, we give a few informal tips for using Theorems 1 and 2 in concrete calculations.

In specific calculations, many of the error terms in Theorems 1 and 2 are zero, and it is convenient to have a quick method for identifying those that are nonzero. To this end, it is useful to keep in mind the underlying geometry, and also to make use of the \emph{deficiencies}, introduced in the proof of Proposition \ref{P:intersection}.  Consider, for example, the reduction recursion:

\begin{align*}
\langle \prod_{i=1}^{n} \psi_i^{d_i} \rangle^{m+1} &=\langle\prod_{i=1}^{n}\psi_i^{d_i} \rangle^{m} +\frac{m! }{24} \sum_{S \in {n \brack m+1}}(-1)^{\star(S)} \prod_{j=1}^{k(S)}{|S_j|-1 \choose \{d_i\}_{i \in S_j}}.
\end{align*}

We know that the error terms correspond to the irreducible components of the exceptional locus of the map
$$
\SM_{1,n}(m) \dashrightarrow \SM_{1,n}(m+1),
$$
i.e. each partition $S:=\{S_1, \ldots, S_m\}$ corresponds to a boundary stratum 
$$\Delta_S \simeq \SM_{1,m+1}(m) \times \SM_{0,|S_1|+1} \times \ldots \times \SM_{1,|S_k|+1}$$ parametrizing curves whose topological type is pictured in Figure 1(b). For each $i=1, \ldots, k$, we have defined the deficiencies
$$
e_i:=|S_i|-2-\sum_{j \in S_i}d_j
$$
in Proposition \ref{P:intersection}. Note that if one restricts $\prod_{j \in S_i}\psi_j^{d_j}$ to $\Delta_S$, then the resulting class is pulled back from $A_{e_i}(\M_{0,|S_i|+1}),$ i.e. $e_i$ measures how many more codimension-one classes are necessary to obtain a top-dimensional class on the $\M_{0,|S_i|+1}$-factor of $\Delta_S$. From this point of view, it is perhaps intuitive that if $e_i<0$ for any $i \in \{1, \ldots, k\}$, then the corresponding error term should vanish. This is true since
$${|S_i|-2 \choose \{d_j\}_{j \in S_i}}=0$$ if and only if $e_i<0$.

On the other hand, if $e_i>0$ for all $i \in \{1, \ldots, k\}$, then the value of the error term is just the degree of the top-dimensional class on $\Delta_S$ obtained by augmenting each class $\prod_{j \in S_i} \psi_j^{d_j}$ by the appropriate power of the $\psi$-class of the $(|S_i|+1)^{st}$ section, and multiplying by any top-dimensional product of $\psi$-classes on $\M_{1,m+1}(m)$, i.e. the error term (up to parity) is given by
$$
\langle \psi_1^{m+1} \rangle^m \langle x_1^{e_1} \prod_{j \in S_1}\psi_j^{d_j} \rangle_{0, |S_1|+1} \ldots \langle x_k^{e_k} \prod_{j \in S_k}\psi_j^{d_j} \rangle_{0,|S_k|+1}=\frac{m! }{24} \prod_{j=1}^{k}{|S_j|-1 \choose \{d_i\}_{i \in S_j}}.
$$
where $x_i$ is the chern class of the cotangent bundle over the $(|S_i|+1)^{st}$ section.\footnote{One way of thinking about this, which emerges from examination of the proof of Proposition \ref{P:intersection}, is that as long as the excess dimension of the class $\prod_{i=1}^n \psi_i^{d_i}$ is concentrated on the $\M_{1,m+1}(m)$ factor, then the fundamental relation in $A^*(E_S)$ allows this dimensional excess to leak over and augment the dimensional deficiencies of the classes $\prod_{j \in S_i}\psi_j^{d_j}$ on the $\M_{0,|S_i|+1}$ factors.} 
As an added bonus, if one has computed the deficiencies associated to a given partition, the parity of the corresponding error term is easily recognized. It is just $(-1)^{\sum_{i=1}^{k} e_i}$, since we have

\begin{align*}
\sum_{i=1}^{k} e_i&=\sum_{i=1}^{k}|S_i|-2k-\sum_{i \in S_1 \cup \ldots \cup S_k} d_i,\\
&=(n-(m+1-k))-2k-\sum_{i \in S_1 \cup \ldots \cup S_k} d_i,\\
&=n-m-k-\sum_{i \in S_1 \cup \ldots \cup S_k} d_i-1\\
&=\star(S).
\end{align*}

In sum, a practical method for evaluating the error terms in the reduction recursion is obtained as follows:
\begin{enumerate}
\item Draw a figure illustrating the combinatorial type of a generic curve in $\Delta_{S}$ for each partition $S$.
\item Determine which labelings of marked points on the given figures produce non-negative deficiencies on each rational component.
\item Determine the absolute value of the error term by interpreting it as an intersection number on $\Delta_S$, and determine its parity by summing the associated deficiencies.
\end{enumerate}
Hopefully, this heuristic will become clear in the examples that follows. Exactly the same procedure can be used to evaluate the error terms in the string/dilaton recursions: Each partition corresponds to a stratum of curves (whose topological type is pictured in Figure 1(a)), and the corresponding error term is nonzero if and only if the associated deficiencies are all nonnegative. The only catch is that, in these cases, we have $\star(S)=\sum_{i=1}^{k}e_i+1.$ Thus, in contrast to the reduction recursion, the parity of the error terms in the string/dilaton recursions is $\sum_{i=1}^{k}e_i+1$.

\begin{table}[!h]
\caption{Witten-Kontsevich numbers for $m<n \leq 6$.}
\centering
\begin{tabular}{c|r r r r r r}
 & $\M_{1,1}$ & $\M_{1,2}(m)$ & $\M_{1,3}(m)$&$\M_{1,4}(m)$&$\M_{1,5}(m)$&$\M_{1,6}(m)$ \\ [0.8ex] 
\hline
$m=0/1$
&$\langle \tau_1 \rangle=1$&$\langle \tau_0\tau_2 \rangle=1$&$\langle \tau_0^2\tau_3 \rangle=1$&$\langle \tau_0^3\tau_4 \rangle=1$&$\langle \tau_0^4\tau_5 \rangle=1$&$\langle \tau_0^5\tau_6 \rangle=1$ \\
&&$\langle \tau_1^2 \rangle=1$&$\langle \tau_0\tau_1\tau_2 \rangle=2$&$\langle \tau_0^2\tau_1\tau_3 \rangle=3$&$\langle \tau_0^3\tau_1\tau_4 \rangle=4$&$\langle \tau_0^4\tau_1\tau_5 \rangle=5$\\
&&&$\langle \tau_1^3 \rangle=2$&$\langle \tau_0^2\tau_2^2 \rangle=4$&$\langle \tau_0^3\tau_2\tau_3 \rangle=7$&$\langle \tau_0^4\tau_2\tau_4 \rangle=11$\\
&&&&$\langle \tau_0\tau_1^2\tau_2 \rangle=6$&$\langle \tau_0^2\tau_1^2\tau_3 \rangle=12$&$\langle \tau_0^4\tau_3^2 \rangle=14$\\
&&&&$\langle \tau_1^4 \rangle=6$&$\langle \tau_0^2\tau_1\tau_2^2 \rangle=16$&$\langle \tau_0^3\tau_1^2\tau_4 \rangle=20$\\
&&&&&$\langle \tau_0\tau_1^3\tau_2 \rangle=24$&$\langle \tau_0^3\tau_1\tau_2\tau_3\rangle=35$\\
&&&&&$\langle \tau_1^5 \rangle=24$&$\langle \tau_0^3\tau_2^3 \rangle=48$\\
&&&&&&$\langle \tau_0^2\tau_1^3\tau_3 \rangle=60$\\
&&&&&&$\langle \tau_0^2\tau_1^2\tau_2^2 \rangle=80$\\
&&&&&&$\langle \tau_0\tau_1^4\tau_2 \rangle=120$\\
&&&&&&$\langle \tau_1^6 \rangle=120$\\
&&&&&&\\
$m=2$
&&&$\langle \tau_0^2\tau_3 \rangle=2$&$\langle \tau_0^3\tau_4 \rangle=0$&$\langle \tau_0^4\tau_5 \rangle=2$ &$\langle \tau_0^5\tau_6 \rangle=0$ \\
&&&$\langle \tau_0\tau_1\tau_2 \rangle=2$&$\langle \tau_0^2\tau_1\tau_3 \rangle=4$&$\langle \tau_0^3\tau_1\tau_4 \rangle=2$&$\langle \tau_0^4\tau_1\tau_5 \rangle=8$\\
&&&$\langle \tau_1^3 \rangle=2$&$\langle \tau_0^2\tau_2^2 \rangle=4$&$\langle \tau_0^3\tau_2\tau_3 \rangle=8$&$\langle \tau_0^4\tau_2\tau_4 \rangle=8$\\
&&&&$\langle \tau_0\tau_1^2\tau_2 \rangle=6$&$\langle \tau_0^2\tau_1^2\tau_3 \rangle=14$&$\langle \tau_0^4\tau_3^2 \rangle=16$\\
&&&&$\langle \tau_1^4 \rangle=6$&$\langle \tau_0^2\tau_1\tau_2^2 \rangle=16$&$\langle \tau_0^3\tau_1^2\tau_4 \rangle=14$\\
&&&&&$\langle \tau_0\tau_1^3\tau_2 \rangle=24$&$\langle \tau_0^3\tau_1\tau_2\tau_3\rangle=38$\\
&&&&&$\langle \tau_1^5 \rangle=24$&$\langle \tau_0^3\tau_2^3 \rangle=48$\\
&&&&&&$\langle \tau_0^2\tau_1^3\tau_3 \rangle=66$\\
&&&&&&$\langle \tau_0^2\tau_1^2\tau_2^2 \rangle=80$\\
&&&&&&$\langle \tau_0\tau_1^4\tau_2 \rangle=120$\\
&&&&&&$\langle \tau_1^6 \rangle=120$\\
&&&&&\\
$m=3$
&&&&$\langle \tau_0^3\tau_4 \rangle=6$&$\langle \tau_0^4\tau_5 \rangle=-12$ &$\langle \tau_0^5\tau_6 \rangle=30$ \\
&&&&$\langle \tau_0^2\tau_1\tau_3 \rangle=6$&$\langle \tau_0^3\tau_1\tau_4 \rangle=6$&$\langle \tau_0^4\tau_1\tau_5 \rangle=-18$\\
&&&&$\langle \tau_0^2\tau_2^2 \rangle=6$&$\langle \tau_0^3\tau_2\tau_3 \rangle=6$&$\langle \tau_0^4\tau_2\tau_4 \rangle=18$\\
&&&&$\langle \tau_0\tau_1^2\tau_2 \rangle=6$&$\langle \tau_0^2\tau_1^2\tau_3 \rangle=18$&$\langle \tau_0^4\tau_3^2 \rangle=18$\\
&&&&$\langle \tau_1^4 \rangle=6$&$\langle \tau_0^2\tau_1\tau_2^2 \rangle=18$&$\langle \tau_0^3\tau_1^2\tau_4 \rangle=18$\\
&&&&&$\langle \tau_0\tau_1^3\tau_2 \rangle=24$&$\langle \tau_0^3\tau_1\tau_2\tau_3\rangle=36$\\
&&&&&$\langle \tau_1^5 \rangle=24$&$\langle \tau_0^3\tau_2^3 \rangle=54$\\
&&&&&&$\langle \tau_0^2\tau_1^3\tau_3 \rangle=78$\\
&&&&&&$\langle \tau_0^2\tau_1^2\tau_2^2 \rangle=84$\\
&&&&&&$\langle \tau_0\tau_1^4\tau_2 \rangle=120$\\
&&&&&&$\langle \tau_1^6 \rangle=120$\\
&&&&&\\
$m=4$
&&&&&$\langle \tau_0^4\tau_5 \rangle=24$ &$\langle \tau_0^5\tau_6 \rangle=-120$ \\
&&&&&$\langle \tau_0^3\tau_1\tau_4 \rangle=24$&$\langle \tau_0^4\tau_1\tau_5 \rangle=-24$\\
&&&&&$\langle \tau_0^3\tau_2\tau_3 \rangle=24$&$\langle \tau_0^4\tau_2\tau_4 \rangle=-24$\\
&&&&&$\langle \tau_0^2\tau_1^2\tau_3 \rangle=24$&$\langle \tau_0^4\tau_3^2 \rangle=-24$\\
&&&&&$\langle \tau_0^2\tau_1\tau_2^2 \rangle=24$&$\langle \tau_0^3\tau_1^2\tau_4 \rangle=48$\\
&&&&&$\langle \tau_0\tau_1^3\tau_2 \rangle=24$&$\langle \tau_0^3\tau_1\tau_2\tau_3\rangle=48$\\
&&&&&$\langle \tau_1^5 \rangle=24$&$\langle \tau_0^3\tau_2^3 \rangle=48$\\
&&&&&&$\langle \tau_0^2\tau_1^3\tau_3 \rangle=96$\\
&&&&&&$\langle \tau_0^2\tau_1^2\tau_2^2 \rangle=96$\\
&&&&&&$\langle \tau_0\tau_1^4\tau_2 \rangle=120$\\
&&&&&&$\langle \tau_1^6 \rangle=120$\\
&&&&&\\
$m=5$
&&&&&&$\langle \tau_1^6 \rangle=120$ \\
[1ex]
\end{tabular}
\label{table:nonlin}
\end{table}

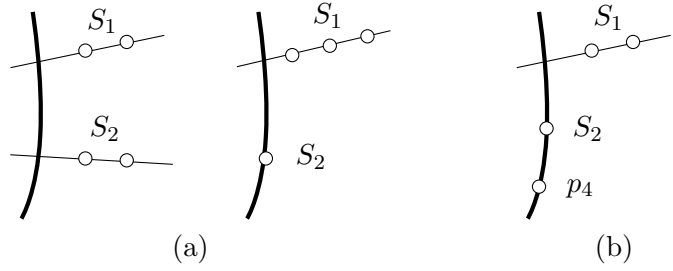
\begin{figure}

\centering
\begin{tikzpicture}[scale=.5]
\draw[ultra thick] plot [smooth, tension=1] coordinates {(1.8,4.3)  (2,1) (1.5,-1.2)} ;
\draw[rotate=5] plot [smooth, tension=1] coordinates {(1.45,2.7)  (5.6,3.2)};
\node at (3.65,4.05) {$S_{1}$};
\draw[fill=white] (4.3,3.5) circle(5pt);
\draw[fill=white] (3.2,3.27) circle(5pt);
\draw[rotate=-10] plot [smooth, tension=1] coordinates {(1.1,.7)  (5.4,1.2)};
\node at (3.72,1.2) {$S_{2}$};
\draw[fill=white] (4.3,.35) circle(5pt);
\draw[fill=white] (3.2,.4) circle(5pt);

\draw[ultra thick] plot [smooth, tension=1] coordinates {(7.8,4.3)  (8,1) (7.5,-1.2)} ;
\draw[rotate=5] plot [smooth, tension=1] coordinates {(7.45,2.2)  (11.6,2.8)};
\node at (9.65,4.2) {$S_{1}$};
\draw[fill=white] (10.7,3.65) circle(5pt);
\draw[fill=white] (9.7,3.4) circle(5pt);
\draw[fill=white] (8.7,3.15) circle(5pt);
\node at (9.2,.45) {$S_{2}$};

\draw[fill=white] (8,.4) circle(5pt);

\node at (6,-2){(a)};
\end{tikzpicture}
\qquad\qquad
\begin{tikzpicture}[scale=.5]
\draw[ultra thick] plot [smooth, tension=1] coordinates {(1.8,4.3)  (2,1) (1.5,-1.2)} ;
\draw[rotate=5] plot [smooth, tension=1] coordinates {(1.45,2.7)  (5.6,3.2)};
\node at (3.65,4.15) {$S_{1}$};
\draw[fill=white] (4.3,3.5) circle(5pt);
\draw[fill=white] (3.2,3.27) circle(5pt);
\node at (3.1,1.2) {$S_{2}$};
\node at (2.9,-.4) {$p_4$};
\draw[fill=white] (2,1.2) circle(5pt);
\draw[fill=white] (1.8,-.35) circle(5pt);

\node at (3.8,-2){(b)};
\end{tikzpicture}
\caption{Combinatorial types of irreducible components of the exceptional locus of  (a) $\M_{1,4}(1) \dashrightarrow \M_{1,4}(2)$, and (b) $\M_{1,4}(2) \dashrightarrow \M_{1,3}(2)$ } \label{F2}
\end{figure}

\subsection{Reduction Recursion}\label{S:Rcompute}
Let us use the reduction recursion to compute Witten-Kontsevich numbers on $\M_{1,4}(2)$ in terms of those on $\M_{1,4}(1)$.
In light of the above discussion, the relevant formula is
$$
\langle \prod_{i=1}^4 \psi_i^{d_i} \rangle^2=\langle  \prod_{i=1}^4 \psi_i^{d_i}\rangle^1+\sum_{S \in {4 \brack 2}}(-1)^{\sum_{i=1}^{k} e_i} \,\prod_{j=1}^{k}\langle \psi_1^2 \rangle^1{|S_j|-1 \choose \{d_i\}_{i \in S_j}}.
$$
The error terms correspond to exceptional components of $\M_{1,4}(1) \dashrightarrow \M_{1,4}(2)$, which come in two types, namely those of types \{2,2\} and type \{1,3\} (see Figure 2(a)).

\subsubsection{$\langle \tau_0^3\tau_4 \rangle^2$}\label{S:calc1}
We evaluate $\langle \tau_0^3\tau_4 \rangle^2$ by computing $\langle \psi_1^4 \rangle^2.$ First, note that the error terms corresponding to partitions of type $\{2,2\}$ will all vanish since $p_1$ will necessarily be located on either $R_1$ or $R_2$, thereby forcing either $e_1$ or $e_2$ to be negative. Similarly, the only stratum of type \{1,3\} that can contribute a nonzero error term is that which has $p_1$ supported on the elliptic component, i.e. the partition $\{\{1\}, \{2,3,4\}\}$. In this case, $e_1=|S_1|-2-\sum_{j \in S_1}d_j=3-2-0=1$, so the corresponding error term is negative, and we get

\begin{align*}
\langle \psi_1^4 \rangle^2&=\langle \psi_1^4 \rangle^1-\langle \psi_1^2 \rangle^1 \langle x_1^4 \rangle_{0,4}\\
&=\langle \psi_1^4 \rangle^1-\langle \psi_1^2 \rangle^1\\
&=1-1=0.\\
\end{align*}

\subsubsection{ $\langle \tau_0^2\tau_1 \tau_3 \rangle^2$}
We evaluate $\langle \tau_0^2\tau_1 \tau_3 \rangle^2$ by computing $\langle \psi_1^3\psi_2 \rangle^2.$ As in \ref{S:calc1}, $p_1$ can't lie on any rational component of any boundary stratum without producing a negative deficiency, so the only stratum that can contribute a nonzero error term the same as in our computation of $\langle \psi_1^4 \rangle^2$, namely $\{\{1\}, \{2,3,4\}\}$. In this case, we have
$e_1=|S_1|-2-\sum_{j \in S_1}d_j=3-2-1=0$, so we get 

\begin{align*}
\langle \psi_1^3\psi_2 \rangle^2&=\langle \psi_1^3\psi_2 \rangle^1+\langle \psi_1^2 \rangle^1 \langle x_1^4 \rangle_{0,4}\\
&=\langle  \psi_1^3\psi_2 \rangle^1+\langle \psi_1^2 \rangle^1\\
&=3+1=4.\\
\end{align*}

\subsubsection{$\langle \tau_0^2\tau_2^2 \rangle^2$, $\langle \tau_0\tau_1^2\tau_2 \rangle^2$, $\langle \tau_1^4 \rangle^2$.}
These 2-stable Witten Kontsevich numbers are all equal to the corresponding 1-stable (or ordinary) Witten-Kontsevich numbers. This is easily seen in terms of the heuristics explained above. In each case, a stratum of type $\{1,3\}$ will necessarily have an associated negative deficiency because the sum of the exponents $d_i$ on any three of the four sections is at least two. Similarly, every stratum of type $\{2,2\}$ will have an associated negative deficiency because any way of dividing the four sections into two pairs produces at least one pair whose exponents $d_i$ sum to at least 2.

\subsection{String/Dilaton Recursion}
Now we use the string and dilaton recursions to compute the Witten-Kontsevich numbers on $\M_{1,4}(2)$ in terms of those on $\M_{1,3}(2)$. We will see that our results are consistent with the calculations in \ref{S:Rcompute}. In light of the preceding discussion, the relevant formulae are
\begin{align*}
\langle \prod_{i=1}^{3} \psi_i^{d_i} \cdot \psi_{4}^{0} \rangle^2&=\sum_{j=1}^{3}\langle \prod_{i=1}^{3}\psi_i^{d_i-\delta_{ij}} \rangle^2 + \langle \psi_1^{3} \rangle^2 \sum_{S \in {3 \brack 2}}(-1)^{\sum_{i=1}^{k}e_i+1}\,\prod_{j=1}^{k}{|S_j|-1 \choose \{d_i\}_{i \in S_j}}.\\
\langle \prod_{i=1}^{3} \psi_i^{d_i} \cdot \psi_{4} \rangle^2 &=3 \langle\prod_{i=1}^{3}\psi_i^{d_i} \rangle^2 +\langle \psi_1^{3} \rangle^2\sum_{S \in {3 \brack 2}}(-1)^{\sum_{i=1}^{k}e_i+1}\, \prod_{j=1}^{k}{|S_j|-1 \choose \{d_i\}_{i \in S_j}}.
\end{align*}
The error terms correspond to exceptional components of the map $\M_{1,4}(2) \dashrightarrow \M_{1,3}(2)$, which are all of the same combinatorial type, namely $\{1,2\}$ (see Figure 2(b)).

\subsubsection{$\langle \tau_0^3\tau_4 \rangle^2$}
To compute $\langle \tau_0^3\tau_4 \rangle^2=\langle \psi_1^4 \rangle^2$, we use the string equation. As in \ref{S:calc1}, if $p_1$ is located on the rational component, then the associated deficiency is negative, so the only nonzero error term corresponds to the partition $\{\{1\}, \{2,3\}\}$. For this partition, the associated deficiency is $e_1=2-2=0$. Thus, the associated error term is negative (recall $\star(S)=\sum_{i=1}^{k}e_i+1$ for the string/dilaton recursions), and we obtain

\begin{align*}
\langle \psi_1^4 \rangle^2&=\langle \psi_1^3 \rangle^2-\langle \psi_1^3 \rangle^2 \langle x_1^3 \rangle_{0,3}\\
&=\langle \psi_1^3 \rangle^2-\langle \psi_1^3 \rangle^2\\
&=2-2=0.\\
\end{align*}

\subsubsection{$\langle \tau_0^2\tau_1 \tau_3 \rangle^2$}
We can evaluate this number either by the string equation (representing it as $\langle \psi_1^3\psi_2 \rangle^2$), or the dilaton equation (representing it as $\langle \psi_1^3\psi_4 \rangle^2$).
Using first the string equation, we see that neither $p_1$ nor $p_2$ can lie on the rational component of a boundary stratum without producing a negative deficiency. Thus, in this case, there are no nonzero error terms, and we obtain
\begin{align*}
\langle \psi_1^3\psi_2 \rangle^2&=\langle \psi_1^3 \rangle^2+\langle \psi_1^2 \psi_2 \rangle^2\\
&=2+2=4.\\
\end{align*}

If we use the dilaton equation, then we get exactly one nonzero error term corresponding to the partition $\{\{1\}, \{2,3\}\}$. The deficiency is $e_1=2-2=0$, so we obtain

\begin{align*}
\langle \psi_1^3\psi_4 \rangle^2&=3 \langle \psi_1^3 \rangle^2-\langle \psi_1^3 \rangle^2 \langle x_1^3 \rangle_{0,3}\\
&=2 \langle  \psi_1^3 \rangle^2\\
&=4.\\
\end{align*}

\subsubsection{$\langle \tau_0^2\tau_2^2 \rangle^2$, $\langle \tau_0\tau_1^2\tau_2 \rangle^2$, $\langle \tau_1^4 \rangle^2$.}
These 2-stable Witten Kontsevich numbers are all equal to the corresponding 1-stable (or ordinary) Witten-Kontsevich numbers.  The first two can be evaluated by the string equation, while the latter two can be evaluated using the dilaton equation. In every case, using the same analysis as above, we see that all the error terms vanish.

\bibliography{references}{}
\bibliographystyle{amsalpha}

\end{document}